\documentclass[11pt,reqno]{amsart}
\usepackage[utf8]{inputenc}
\usepackage{amsmath,amssymb,amsthm,mathrsfs,mathtools,color,dsfont,enumerate}
\usepackage{a4wide}
\usepackage{graphics}
\usepackage{epsfig}
\usepackage{caption}
\usepackage{subcaption}
\usepackage{enumitem}   
\usepackage[colorlinks=true]{hyperref}
\hypersetup{,urlcolor=blue, citecolor=red, linkcolor=blue}
\usepackage[]{diagbox}
\usepackage{bm} 

\newtheorem{theorem}{Theorem}[section]
\newtheorem{prop}[theorem]{Proposition}
\newtheorem{lemma}[theorem]{Lemma}
\newtheorem{corollary}[theorem]{Corollary}
\newtheorem{definition}[theorem]{Definition}
\newtheorem{remark}[theorem]{Remark}
\newtheorem{ex}[theorem]{Example}

\newcommand{\R}{{\mathbb R}}
\newcommand{\Z}{{\mathbb Z}}
\newcommand{\N}{{\mathbb N}}
\newcommand{\E}{{\mathbb E}}

\renewcommand{\P}{{\mathbb P}}

\newcommand{\pol}{\textup{pol}}

\newcommand{\USC}{\textsc{USC}}
\newcommand{\LSC}{\textsc{LSC}}
\newcommand{\bmcP}{\overline{\mathcal{P}}}

\def\bbar{\overline}
\def\brho{\bbar \rho}

\def\hx{{\hat x}}

\def\hf{{\hat f}}

\def\hX{\hat X}
\def\hY{\hat Y}

\def\bH{\overline H}

\def\tf{{\tilde f}}

\def\hu{\hat u}

\def\mcA{{\mathcal A}}
\def\mcC{{\mathcal C}}
\def\mcD{{\mathcal D}}
\def\mcE{{\mathcal E}}

\def\mcK{{\mathcal K}}
\def\mcL{{\mathcal L}}

\def\mcM{{\mathcal M}}

\def\mcO{{\mathcal O}}
\def\mcP{{\mathcal P}}
\def\mcR{{\mathcal R}}
\def\mcS{{\mathcal S}}
\def\mcX{{\mathcal X}}
\def\mcY{{\mathcal Y}}
\def\bmcO{\bar{{\mathcal O}}}

\def\bDf{\overline{D_f}}

\def\a{\alpha}

\def\<{\langle}
\def\>{\rangle}

\def \pol{{\mathbf{pol}}}
\newcommand\ind[1]{\mathds{1}_{#1}}
\newcommand{\Dx}{{\Delta x}}
\newcommand{\PhDx}{\Pi^h_\Dx}

\makeatletter
 \@addtoreset{equation}{section}
 \makeatother
 
\allowdisplaybreaks
\usepackage{appendix}
\usepackage{hyperref}

\title[Some PDE results in Heston model with applications]{Some PDE results in Heston model with applications}
\date{\today}
\author{Edoardo Lombardo}
\address{Edoardo Lombardo, Universit\`a degli Studi di Roma Tor Vergata, Rome, Italy. }
\email{edoardolombardo92@gmail.com}


\begin{document}

\begin{abstract}
  We present here some results for the PDE related to the logHeston model. We present different regularity results and prove a verification theorem that shows that the solution produced via the Feynman-Kac Theorem is the unique viscosity solution for a wide choice of initial data (even discontinuous) and source data. In addition, our techniques do not use Feller's condition at any time.
  In the end, we prove a convergence theorem to approximate this solution by means of a hybrid (finite differences/tree scheme) approach. 
\end{abstract}

\maketitle

\section{Introduction}
The stochastic volatility model proposed by Heston in \cite{Heston} is one of the most famous and used models in mathematical finance. It describes the evolution of the price of an asset $S$ and its instantaneous volatility $Y$, according to the following couple of stochastic differential equations
\begin{equation}\label{Heston_model_sde_intro}
  \begin{split}
 dS_t &= (r-\delta)S_t dt+  S_t\sqrt{Y_t} (\rho dW_t + \brho dB_t),\\
 dY_t &= (a-bY_t)dt + \sigma\sqrt{Y_t}dW_t,
  \end{split}
\end{equation}
where $r\in\R$, $\delta\ge0$, $a,b,\sigma>0$, $\rho\in(-1,1)$, $\brho=\sqrt{1-\rho^2}$, $(x,y)\in\R\times[0,\infty)$ and $(W,B)$ is a standard 2-dimensional Brownian motion.
In order to study and discretize the asset $S$, it is useful to consider the logHeston diffusion obtained by applying the transformation $(s,y)\mapsto(\log(s),y)$ to the asset and the volatility. To this purpose, we consider a slightly general model from which we can recover the logHeston by a precise choice of the parameters:
\begin{equation}\label{logHeston_model_sde_intro}
  \begin{split}
 dX_t&= (c+d Y_t)dt+\lambda\sqrt{Y_t}(\rho dW_t + \brho dB_t)\\
 dY_t&= (a-bY_t)dt+\sigma\sqrt{Y_t}dW_t,
  \end{split}
  \end{equation}
where $b,c,d\in\R$, $a,\lambda,\sigma>0$ $\rho\in(-1,1)$. Indeed, one can show that the price of financial derivatives written on $(X,Y)$, or equivalently on $(X,Y)$, satisfies a peculiar parabolic PDE that is degenerate, i.e. the matrix of the second-order derivative coefficients fails to be strictly positive definite when the boundary $\{y=0\}$ is attained. 
In this case, the classical existence and uniqueness results using the uniform ellipticity property fail, and an ad hoc method must be found to prove them.

The literature presents various existence and uniqueness results derived from analyzing the Heston and logHeston PDEs. Ekström and Tysk \cite{ET2010} examined a PDE arising from a generalized Heston model. While their model employs more general drift and volatility functions, these retain key characteristics of the original ones, such as positive drift of the volatility process when $Y_t=0$ and sufficient regularity of the squared volatility function. Within their model, they establish a verification theorem and a uniqueness result contingent upon certain mild assumptions on the payoff function $f$.
Costantini et al. groundbreaking study in \cite{CPD12} establish the existence and uniqueness of a viscosity solution for a PDE encompassing a wide array of jump-diffusion models, including Heston, applicable to both European and Asian options. Notably, when applied to the Heston model, their innovative approach necessitates a condition akin to the Feller condition (i.e. $\sigma^2\le 2a$), ensuring the volatility process remains strictly positive. Consequently, their result cannot cover the full range of parameters in the Heston case.
For numerical reasons, Briani et al. \cite{BCT} need regularity results for functional of the diffusion $(X^{t,x,y},Y^{t,y})$, representing the solution of \eqref{logHeston_model_sde_intro} starting from $(x,y)$ at time $t\in[0,T)$. This is important because the expectation of such functionals gives the price of European options.
In order to do that, they prove first a verification result (cf. \cite[Lemma 5.7]{BCT}) for the logHeston PDE \eqref{reference_PDE_generalset}: under appropriate regularity hypotheses on $f$ and $h$, the function
\begin{equation}\label{u_sol_gen_intro}
 u(t,x,y)=\E\bigg[e^{\varrho (T-t)}f(X^{t,x,y}_T,Y^{t,y}_T)-\int_t^Te^{\varrho (s-t)} h(s,X^{t,x,y}_s,Y^{t,y}_s)ds\bigg]
\end{equation}
is a solution of  
\begin{equation}\label{reference_PDE_generalset}
  \begin{cases}
    \partial_tu(t,x,y) +\mcL u(t,x,y) +\varrho u(t,x,y) = h(t,x,y),\quad   t\in[0,T), &(x,y)\in\mcO, \\
 u(T,x,y) = f(x,y), \quad  &(x,y)\in\mcO,
  \end{cases}
\end{equation}
where $\mcO=\R\times(0,\infty)$ and $\mcL$ is the infinitesimal generator of \eqref{logHeston_model_sde_intro}:
\begin{equation}\label{reference_infinit_gen}
\mcL = \frac{y}{2}(\lambda^2\partial^2_{x} + 2 \rho\lambda\sigma \partial_{x}\partial_{y} +  \sigma^2 \partial^2_{y}) +(c+dy) \partial_x + (a-by) \partial_y.
\end{equation}
Moreover, when the Feller condition $\sigma^2\le2a$ is satisfied, the uniqueness of the solution holds. \label{Page:Feller_condition_uniqueness} In fact, since the boundary $\R\times\{0\}$ is inaccessible to the process $(X,Y)$ under the Feller condition, the behaviour of $u$ if $\R\times\{y=0\}$ is irrelevant and one can achieve uniqueness of the solution. In second instance, they prove a stochastic representation for the derivatives of $u$. As a consequence of this result, and under specific conditions on final data $f$, they show that $u$ is regular enough to solve, by continuity, the problem even on $\overline{\mcO}=\R\times[0,\infty)$, so the PDE is satisfied even when volatility collapses in 0. This gives an additional boundary condition, giving an equation involving the function and its derivatives at the domain boundary. It is worth noting that this is called a Robin boundary condition.
Briani et al. did not establish uniqueness for the PDE over $\overline{\mcO}$, as their primary objective was different. 

In this paper we restart from the logHeston setting of \cite{BCT} and study minimal hypotheses over $f$ and $h$ under which we can prove a verification theorem that characterizes $u$ in \eqref{u_gen_sol_intro} as the \textit{unique} solution of \eqref{reference_PDE_generalset} even when the Feller condition is not met ($\sigma^2>2a$).
Let us stress that, due to the connection between $u$ in \eqref{u_sol_gen_intro} and option prices problems, weaker requests on $f$ and $h$ translate into the choice of more general payoffs and $h$ running costs in finance.
First, we consider classical solutions. 
To this purpose, in Section \ref{classical_section}, after reviewing a slight extension of the regularity result obtained in \cite{BCT}, we characterize $u$ as the unique classical solution of \eqref{reference_PDE_generalset} over $\overline{\mcO}=\R\times[0,\infty)$ under relaxed hypotheses on $f$ and $h$. However, this requires that $f$ and $h$ have some differentiability conditions: merely continuity properties are not enough. In order to contour this difficulty, solutions in weak or viscosity sense are typically taken into account. Here, in Section \ref{viscosity_section}, we tackle the problem from a viscosity solutions point of view: we prove an existence and uniqueness result without the restriction of the Feller condition. We stress that the initial data may have some types of discontinuities, allowing us to deal with valuable financial examples such as Digital options.

We point out that, under the Feller condition, uniqueness results for classical, viscosity or weak solutions over $\mcO$ or $\overline{\mcO}$ have already been studied in the literature (see, e.g. \cite{BCT,CPD12,CMA17}), possibly requiring strong hypotheses on $f$ and $h$.
However, when the Feller condition does not hold, to the best of our knowledge, the literature is very poor on results concerning the uniqueness of classical and viscosity solutions (see, e.g. \cite{ET2010} for classical solutions). Thus, the main original contributions of this paper delve into this direction.

Finally, we deal with the convergence of a hybrid numerical method introduced in \cite{BCZ}.
If $f$ is smooth enough, the convergence rate has already been provided in Briani et al. \cite{BCT}, independently of the validity of the Feller condition. As the authors need the regularity of the price function, they strongly use classical solution results. Thus, their approach must keep the regularity of $f$. Here, by exploiting the tools and techniques introduced to get the results concerning viscosity solutions, we can prove the convergence of the above-cited hybrid numerical scheme whenever $f$ is continuous, see Theorem \ref{convergence_theorem}. The result of this theorem is confirmed empirically by the numerical experiment carried out in \cite{BCZ}, which computes the price of a European put option in the Heston model. Other numerical experiments that use the hybrid algorithm for the Bates and for the Heston-Hull-White models have been carried out in \cite{BCTZ} and \cite{BCZ2} respectively.

\section{Existence and uniqueness of classical solutions}\label{classical_section}
This section contains a slight improvement of some results proven in \cite{BCT} regarding the log-Heston PDE in which we add a uniqueness result inspired by \cite{ET2010}.

We start by introducing some notations. We set $\R_+=[0,\infty)$, $\R^*_+=(0,\infty)$ and name $\mcC^{q}(\R\times\R_+)$ the set of all functions on $\R\times\R_+$ which are $q$-times continuously differentiable. We set $\mcC_{\pol}^{q}(\R\times\R_+)$ the set of functions $g \in \mcC^{q}(\R\times\R_+)$ such that there exist $C, L>0$ for which
$$ 
|\partial_{x}^{\alpha}\partial_{y}^{\beta} g(x,y)| \leq C(1+|x|^{L}+y^{L}), \quad (x,y) \in \R\times\R_+,\; \alpha+\beta \leq q .
$$
For $T>0$, we set $\mcC_{\pol, T}^{q}(\R\times\R_+)$ the set of functions $v=v(t, x, y)$ such that $v \in \mcC^{\lfloor q / 2\rfloor, q}([0, T) \times (\R\times\R_+))$ and there exist $C, L>0$ for which
$$ 
\sup _{t \in[0, T)}|\partial_{t}^{k}\partial_{x}^{\alpha}\partial_{y}^{\beta}  v(t, y)| \leq C(1+|x|^{L}+y^{L}), \quad (x,y) \in \R\times\R_+,\; 2 k+\alpha+\beta \leq q. 
$$
We set $\mcC(\R\times\R_+)=\mcC^{0}(\R\times\R_+)$, $\mcC_{\pol}(\R\times\R_+)=\mcC_{\pol }^{0}(\R\times\R_+)$ and $\mcC_{\pol,T}(\R\times\R_+)=\mcC_{\pol,T}^{0}(\R\times\R_+)$. We also need another functional space, that we call $\mcC_{\pol}^{p, q}(\R, \R_+), p \in[1, \infty], q \in \N, m \in \N^{*}: g=g(x, y) \in \mcC_{\pol}^{p, q}(\R,\R_+)$ if $g \in \mcC_{\pol}^{q}( \R\times\R_+)$ and there exist $C, c>0$ such that
$$ 
|\partial_{x}^{\alpha} \partial_{y}^{\beta} g(\cdot, y)|_{L^{p}(\R, dx)} \leq C(1+|y|^{c}), \quad\alpha+\beta \leq q,
$$
where $|h|_{L^p,dx}=(\int_\R h(x)^pdx)^{1/p}$ if $p>1$, and the standard sup norm if $p=\infty$.
Similarly, as above, we set $\mcC_{\pol,T}^{p, q}(\R, \R_+)$ the set of the function $v \in \mcC_{\pol,T}^{q}(\R \times \R_+)$ such that
$$ 
\sup _{t \in[0, T)}|\partial_{t}^{k} \partial_{x}^{l'} \partial_{y}^{l} v(t, \cdot, y)|_{L^{p}(\R, dx)} \leq C(1+|y|^{c}), \quad 2 k+|l'|+|l| \leq q. 
$$

We call the solution of the logHeston SDE \eqref{logHeston_model_sde_intro}
\begin{align}\label{referenceDiffusion}
 X_T^{t,x,y} &= x + \int_t^T\big(c + d Y^{t,y}_s\big) ds + \int_t^T\lambda\rho\sqrt{Y^{t,y}_s} dW_s +\int_t^T\lambda\bar{\rho}\sqrt{Y^{t,y}_s} dB_s, \nonumber\\
 Y^{t,y}_T &= y + \int_t^T(a-bY^{t,y}_s)ds + \int_t^T\sigma\sqrt{Y^{t,y}_s}dW_s.
\end{align}
We define the candidate solution
\begin{equation}\label{u_gen_sol}
 u(t,x,y)=\E\bigg[e^{\varrho (T-t)}f(X^{t,x,y}_T,Y^{t,y}_T)-\int_t^Te^{\varrho (s-t)} h(s,X^{t,x,y}_s,Y^{t,y}_s)ds\bigg],
\end{equation}
to the reference PDE
\begin{equation}\label{reference_PDE}
  \begin{cases}
    \partial_tu(t,x,y) +\mcL u(t,x,y) +\varrho u(t,x,y) = h(t,x,y),\quad   t\in[0,T), &(x,y)\in\R\times\R_+, \\
 u(T,x,y) = f(x,y), \quad  &(x,y)\in\R\times\R_+,
  \end{cases}
\end{equation}
where $\mcL$ is defined in \eqref{reference_infinit_gen}.
One can remark that when $c=r-\delta$ (interest rate minus dividend rate) and $d=-\frac 12$, then $(X,Y)$ is the standard logHeston model for the log-price and volatility. When instead $\rho=0$, $c=r-\delta-\frac \rho{\sigma}a$ and $d=\frac \rho{\sigma}b-\frac 12$ and $\lambda = \brho$, we recover a formulation that will be useful to discretize the solution $u$ in Section \ref{approximation_section}.

In order to present the main contribution in this section, we present a slight extension of a regularity result presented in \cite{BCT}, that can be summarized Lemma \ref{lemma-reg} and Proposition \ref{prop-reg-new}.

\begin{lemma}\label{lemma-reg}
 Let $u$ be defined in \eqref{u_gen_sol}, with $f$ and $h$ such that, as $j=0,1$, $\partial_x^{2j}g\in C^{1-j}_{\pol}(\R\times\R_+)$,  $\partial_x^{2j}h\in \mathcal{C}^{1-j}_{\pol,T}(\R\times\R_+)$ along with $
 h$ and $\partial_y h$ locally Hölder continuous in $[0,T)\times\R\times\R_+^*$.
 Then 
  $\partial^{2j}_xu\in \mathcal{C}^{1-j}_{\pol,T}(\R\times\R_+)$ for $j=0,1$, and one has for 
  \begin{align}
    \partial^m_x u(t,x,y) &=\E\left[e^{\varrho (T-t)} \partial^m_x g(X^{t,x,y}_T,Y^{t,x,y}_T)-\int_t^Te^{\varrho (s-t)} \partial^m_x h(s,X^{t,x,y}_s,Y^{t,x,y}_s)ds  \right],\,\, m=1,2,\label{u-x-xx}\\
    \partial_y u(t,x,y) &=\E\bigg[e^{(\varrho-b) (T-t)} \partial_yg(X^{1,t,x,y}_T,Y^{1,t,x,y}_T)\nonumber \\
  &\qquad\qquad+\int_t^Te^{(\varrho-b) (s-t)} \Big[\frac \lambda2 \partial^2_xu+d\partial_xu-\partial_yh\Big](s,X^{1,t,x,y}_s,Y^{1,t,x,y}_s)ds  \bigg],\label{u-y}
  \end{align}
 where
  $(X^{1,t,x,y}_s,Y^{1,t,x,y}_s)$ solves  \eqref{logHeston_model_sde_intro} with new parameters   $\rho_1=\rho$, $c_1=c+\rho\lambda\sigma$, $d_1=d$, $\lambda_1=\lambda$,  $b_1=b$, $a_1=a+\frac{\sigma^2}{2}$, $\sigma_1=\sigma$. Furthermore, $v=\partial_y u$ is the unique solution to the following PDE
  \begin{equation}
    \begin{cases}
 \big[(\partial_t+\mcL_1+\varrho-b)v +(\lambda^2/2 \partial_x^2+d\partial_x)u-\partial_yh\big](t,x,y)=0,\quad  t\in[0,T), \!\!\!\!&(x,y)\in\R\times\R_+^*,  \\
 v(T,x,y) = \partial_y f(x,y), \quad  \!\!\!&(x,y)\in\R\times\R_+^*,
    \end{cases}
  \end{equation}
 where $\mcL_1=\frac{y}{2}(\lambda^2\partial^2_{x} + 2 \rho\sigma \partial_{x}\partial_{y} +  \sigma^2 \partial^2_{y}) +(c +\rho\lambda\sigma+d y) \partial_x + (a+\sigma^2/2 -by) \partial_y$.
\end{lemma}

Iterating this Lemma, it is possible to prove the following result.

\begin{prop} \label{prop-reg-new}
 Let $q\in\N$. For every $j=0,1,\ldots,q$, $\partial_x ^{2j}f\in \mcC^{q-j}_{\pol}(\R\times\R_+)$, $\partial_x ^{2j}h\in \mcC^{q-j}_{\pol, T}(\R\times\R_+)$, and for all $(m,n)\in\N^2$ such that $m+2n\le 2q$ and $m\le 2(q-1)$, $\partial^m_x\partial^n_y h$ locally Hölder continuous in $[0,T)\times\R\times\R_+^*$.
 Let $u$ as in \eqref{u_gen_sol}.
  
 Then $\partial_x ^{2j}u\in \mcC^{q-j}_{\pol, T}(\R\times\R_+)$ for every $j=0,1,\ldots,q$. 
 Moreover, the following stochastic representation holds: for $m+2n\leq 2q$,
  \begin{equation}\label{stoc_repr-new_chap4}
  \begin{split}
  &\partial^{m}_x\partial^{n}_yu(t,x,y)=\E\left[e^{(\varrho-nb) (T-t)} \partial^m_x\partial^n_yf(X^{n,t,x,y}_T,Y^{n,t,x,y}_T)\right]\\
  &\quad+
 \E\left[\int_t^Te^{(\varrho-nb) (s-t)}\left[n\Big(\frac \lambda 2 \partial^{m+2}_x\partial^{n-1}_yu+d\partial^{m+1}_x\partial^{n-1}_yu\Big)-\partial^{m}_x\partial^n_y h\right](s,X^{n,t,x,y}_s,Y^{n,t,x,y}_s)ds  \right],
  \end{split}
  \end{equation}
 where $\partial^{m}_x\partial^{n-1}_y u:=0$ when $n=0$ and $(X^{n,t,x,y},Y^{n,t,x,y})$, $n\geq 0$, denotes the solution starting from $(x,y)$ at time $t$ to the SDE \eqref{logHeston_model_sde_intro} with parameters
  \begin{equation}\label{parameters-new_chap4}
  \rho_n=\rho,\quad c_n=c+n\rho\lambda\sigma,\quad d_n=d,\quad \lambda_n=\lambda,\quad a_n=a+n\frac{\sigma^2}{2}, \quad b_n=b,\quad \sigma_n=\sigma.
  \end{equation}
 In particular, if $q\geq 2$ then $u\in \mcC^{1,2}([0,T]\times \R\times\R_+)$,  solves the PDE
  \begin{equation}\label{PDE-closed}
  \begin{cases}
  \partial_t u(t,x,y)+ \mcL u(t,x,y) +\varrho u(t,x,y)= h(t,x,y),\qquad& t\in [0,T), \,(x,y)\in \R\times\R_+,\\
 u(T,x,y)= f(x,y),\qquad& (x,y)\in \R\times\R_+.
  \end{cases} 
  \end{equation}
\end{prop}

\begin{remark}
 For discretization purposes, as done in Briani et al. \cite{BCT}, one can consider an $L^p$ property for $x\mapsto u(t,x,y)$, and similarly for the derivatives. In this case, one can reformulate Proposition \ref{prop-reg-new} as follows. Let $p\in [1, \infty]$, $q\in\N$. For every $j=0,1,\ldots,q$, $\partial_x ^{2j}f\in \mcC^{p,q-j}_{\pol}(\R,\R_+)$, $\partial_x ^{2j}h\in \mcC^{p,q-j}_{\pol, T}(\R,\R_+)$, and for all $(m,n)\in\N^2$ such that $m+2n\le 2q$ and $m\le 2(q-1)$, $\partial^m_x\partial^n_y h$ locally Hölder continuous in $[0,T)\times\R\times\R^*_+$.
 Then $\partial_x ^{2j}u\in \mcC^{p,q-j}_{\pol, T}(\R,\R_+)$ for every $j=0,1,\ldots,q$. Moreover, the stochastic representation \eqref{stoc_repr-new} holds and, if $q\geq 2$, $u$ solves PDE \eqref{PDE-closed}.
\end{remark}

It is possible to prove these results with the exact proofs presented in \cite{BCT}, with little changes due to the presence of the source term $h$, so we omit the proofs here.

One could be interested to see if we can ask for less regular $f$ and $h$ and still have existence of a classical solution and in which case the solution is unique. In order to do that we first fix other notations that will be required in what follows.

Let $T>0$ and a convex domain $\mcD\subset[0,T]\times\R^{m}$ and $P_1=(t_1,z_1),P_2=(t_2,z_2)\in\mcD$ we define the ``parabolic'' distance $d_\mcP:\mcD\times\mcD\rightarrow\R_+$ as
\begin{equation}\label{parabolic_distance}
 d_\mcP(P_1,P_2) = \big(|t_1-t_2|+|z_1-z_2|^2\big)^{1/2}.
\end{equation}
Let $v:\mcD\rightarrow\R$, using the notation $|v|^\mcD_0 = \sup_{P\in \mcD} |v(P)|$, we introduce the following notation of the $\alpha$-Hölder norm. For $\alpha\in(0,1)$:
\begin{equation}\label{alfa_holder_norm}
 \overline{|v|}^\mcD_\a = |v|^\mcD_0 + \bH^{\mcD}_\a(v),\quad \text{ where } \quad \bH^{\mcD}_\a(v) = \sup_{P\neq Q\mid P,Q\in \mcD}\frac{|v(P)-v(Q)|}{d_\mcP(P,Q)^\a} 
\end{equation}
We will say that $v$ is $\alpha$-Hölder for the parabolic distance if $\bH^{\mcD}_\a(u)<\infty$ that is equivalent to say that $v=v(t,z)$ is $\alpha/2$-Hölder in $t$ and $\alpha$-Hölder in $z$.
We define the $(2+\alpha)$-Hölder norm 
\begin{equation}
 \overline{|v|}^\mcD_{2+\a} =\overline{|v|}^\mcD_\a  +\overline{|\partial_t v|}^\mcD_\a  +\sum_{1\le |l|\le 2} \overline{|\partial^l_z v|}^\mcD_\a.
\end{equation}
To define the weighted $\alpha$-Hölder norm, we must first introduce the notion of weight.
We call for $\tau>0$ and $Q_i=(\tau_i,z)$ $i=1,2$
\begin{equation}\label{weights_for_norm}
  \partial\mcD_\tau=\{(t,z)\in \partial\mcD  \mid t\in[0,\tau]\},\quad d_{Q_i} = \inf_{P\in\partial\mcD_{\tau_i}}d_\mcP(Q_i,P) \text{ and } d_{Q_1,Q_2}=\min(d_{Q_1},d_{Q_2}),
\end{equation}
where $d_{Q_i}$, $i\in\{1,2\}$, measures the parabolic distance of $Q_i$ from the boundary $\partial\mcD_{\tau_i}$. 
Similarly to \eqref{alfa_holder_norm}, for $m\in\N$, $\alpha\in(0,1)$ we define
\begin{equation}\label{weighted_alfa_holder_norm_1}
 |v|^\mcD_{\a,m} = |v|^\mcD_{0,m} + H^{\mcD}_{\a,m}(v),
\end{equation}
where
\begin{equation}\label{weighted_alfa_holder_norm_2}
 |v|^\mcD_{0,m}= \sup_{P\in \mcD} d_{P}^m |v(P)|, \quad H^{\mcD}_{\a,m}(v) = \sup_{P\neq Q\mid P,Q\in \mcD}d_{P,Q}^{m+\a}\frac{|v(P)-v(Q)|}{d_\mcP(P,Q)^\a}.
\end{equation}
With the notation $|v|^\mcD_\a=|v|^\mcD_{\alpha,0}$ weighted $(2+\alpha)$-Hölder norm is as follows
\begin{equation}
 |v|^\mcD_{2+\a} =|v|^\mcD_\a  + |\partial_t v|^\mcD_{\a,1}  +\sum_{1\le |l|\le 2} |\partial^l_z v|^\mcD_{\a,|l|},
\end{equation}
where $l=(l_1,\ldots,l_m)$, $\partial^l_z=\partial^{l_1}_{z_1}\cdots \partial^{l_m}_{z_m}$ and $|l|=\sum_{i=1}^m l_i$.
The main difference between the standard and the weighted $\alpha$-Hölder norm is that the latter allows explosions for the derivatives of $v$ and the difference $|v(P)-v(Q)|$ when we evaluate them near the boundary.

Now we can state the following result,  which clarifies the behavior of the second order spatial derivatives of the solution $u$ near the spatial boundary $\R\times\{0\}$, in which we place hypotheses slightly stronger than the ones in Lemma \ref{lemma-reg}.

\begin{prop}\label{min_hp_pde_existence_sol}
 Let $u$ as in \eqref{u_gen_sol}, $f$ and $h$ such that, for $j=0,1$, $\partial_x^{2j}f\in \mcC^{1-j}_{\pol}(\R\times\R_+)$,  $\partial_x^{2j}h\in \mcC^{1-j}_{\pol,T}(\R\times\R_+)$. Furthermore, we take $|h|^K_{\alpha,2}$, $|\partial_y h|^K_{\alpha,2}<\infty$, for all $K$ convex compact set contained in $[0,T)\times\R\times\R_+$.
Then, for all $t_0\in(0,T)$ and $x_0\in\R$ one has
\begin{equation}\label{0lims_dyyu_dxyu}
 \lim_{(t,x,y)\rightarrow (t_0,x_0,0)} y \partial^2_y u(t,x,y)=0 \quad \text{ and } \quad \lim_{(t,x,y)\rightarrow (t_0,x_0,0)} y \partial_x\partial_y u(t,x,y)=0.
\end{equation}
\end{prop}
\begin{proof}
 The proof takes inspiration from \cite{ET2010}. Under these hypotheses, as shown in Lemma \ref{lemma-reg}, we know that $v=\partial_y u$ solves 
  \begin{equation}
    \begin{cases}
 \big[(\partial_t+\mcL_1+\varrho-b)v +(\lambda^2/2 \partial_x^2+d\partial_x)u-\partial_yh\big](t,x,y)=0,\quad  t\in[0,T), \!\!\!\!&(x,y)\in\R\times\R_+^*,  \\
 v(T,x,y) = \partial_y f(x,y), \quad  \!\!\!&(x,y)\in\R\times\R_+^*,
    \end{cases}
  \end{equation}
 where $\mcL_1=\frac{y}{2}(\lambda^2\partial^2_{x} + 2 \rho\sigma \partial_{x}\partial_{y} +  \sigma^2 \partial^2_{y}) +(c +\rho\lambda\sigma+d y) \partial_x + (a+\sigma^2/2 -by) \partial_y$. 
 We consider, now, a sequence $(t_n,x_n,y_n)_{n \in{\N^*}}\subset [0,T)\times\R\times\R^*_+$ converging to $(t_0,x_0,0)$, where $t_0\in[0,T)$ and $x_0\in\R$. By the convergence $y_n\rightarrow0$, there exists $n_0$ such that for all $n\ge n_0$ $y_n\in(\frac{1}{m_n},\frac{2}{m_n})$, and $m_n\rightarrow\infty$ when $n\rightarrow\infty$. Then we define
  \begin{align*}
    \chi_n:(t,x,y)&\longmapsto(s,\eta,\zeta)=(m_n(t-t_n),m_n(x-x_n),m_ny)
  \end{align*}
 and the functions $w_n$ as
  \begin{equation*}
 w_n(s,\eta,\zeta)= v(\chi_n^{-1}(s,\eta,\zeta))=v\Big(\frac{s}{m_n}+t_n, \frac{\eta}{m_n}+x_n, \frac{\zeta}{m_n}\Big).
  \end{equation*}
 One can check that $w_n$ satisfies the following PDE
  \begin{multline}\label{w_pde}
    \partial_s w_{n}+\frac{\zeta}{2}(\lambda \partial_\eta^2 +2\rho\lambda\sigma\partial_\eta \partial_\zeta+\sigma^2\partial_\zeta^2) w_{n} +(c+\rho\lambda\sigma+d\frac{\zeta}{{m_n}})\partial_\eta w_{n}\\
 +(a+\frac{\sigma^2}{2}+b\frac{\zeta}{{m_n}})\partial_\zeta w_{n} +(\varrho-b)w_{n} + \frac{1}{{m_n}} g_{n} =0,
  \end{multline}
 where
  $$
 g_{n}(s,\eta,\zeta)=(\lambda^2/2 \partial_x^2+d\partial_x)u(\chi_n^{-1}(s,\eta,\zeta)) + \partial_y h(\chi_n^{-1}(s,\eta,\zeta)).
  $$
 We also define 

  $$
 u_n = u\circ \chi^{-1}_n,\quad h_n= h \circ\chi^{-1}_n
  $$

 Then, we consider the rectangle $R_n =[t_n-\frac{2\delta}{m_n},t_n+\frac{2\delta}{m_n}]\times[x_n-\frac{2}{m_n},x_n+\frac{2}{m_n}]\times[\frac{1}{2m_n},\frac{4}{m_n}]$, $(t_n,x_n,y_n)$, and we define $\mcR = \chi_n R_n=[-2\delta,2\delta]\times[-2,2]\times[1/2,4]$.
 By \eqref{w_pde} in $\mcR$ thanks to the Schauder interior estimates (cf. Theorem 5 in Sec.2 of Chap. 3 \cite{AF_book}), one has
  \begin{align*}
 |w_n|^{\mcR}_{2+\alpha} &\le |w_n|^{\mcR}_0 +\frac{1}{m_n}|g_n|^{\mcR}_{\a,2} \\
    &\le |w_n|^{\mcR}_0 +\frac{1}{m_n}\Big(\frac{\lambda^2}{2}|\partial^2_x u_n|^{\mcR}_{\a,2}+|d||\partial_x u_n|^{\mcR}_{\a,2}+|\partial_y h_n|^{\mcR}_{\a,2}\Big) \\
    &\le |v|^{R_n}_0 +\frac{1}{m_n}\Big(\frac{\lambda^2}{2}|\partial^2_x u|^{R_n}_{\a,2}+|d||\partial_x u|^{R_n}_{\a,2}+|\partial_y h|^{R_n}_{\a,2}\Big) \\
    &\le |v|^{R_n}_0 +\frac{C}{m_n}\Big(|u|^{R_n}_0+|h|^{R_n}_{\a,2}+|\partial_y h|^{R_n}_{\a,2}\Big)\\
    &\le |v|^{R_n}_0 +\frac{C}{m_n}\Big(|u|^{K}_0+|h|^{K}_{\a,2}+|\partial_y h|^{K}_{\a,2}\Big)<\hat{C}<\infty,
  \end{align*}
 where we pass from the third to fourth line using in succession: that exist $C_1>0$ such $|\partial_x u|^{R_n}_{\a,2}\le C_1 |\partial_x u|^{R_n}_{\a,1}$, that we can upper bound thanks to Schauder interior estimates on $u$ the two seminorms $|\partial_x u|^{R_n}_{\a,1}$, $|\partial^2_x u|^{R_n}_{\a,2}$ as follows
  $$
 |\partial_x u|^{R_n}_{\a,1} + |\partial^2_x u|^{R_n}_{\a,2} \le |u|^{R_n}_{2+\a}\le |u|^{R_n}_{0} + |h|^{R_n}_{\a,2}.
  $$
 The passage to the fifth line is because there exist $\hat{n}\in\N$ and $K$ compact set such that $R_n\subset K$ for all $n\ge\hat{n}$. The uniform bound by $\hat{C}$ follow by $|u|^{R_n}_0\rightarrow |v(t_0,x_0,0)|$.
 Now that we have the weighted holder norm estimate $|w_n|^{\mcR}_{2+\alpha}<\hat{C}$, we can consider a smaller rectangle $\mcR' = [\delta,\delta]\times[-1,1]\times[1,2]$ who has strictly positive distance from $\partial \mcR$ and get, in terms of standard holder norms, the uniform bound $\overline{|w_n|}^{\mcR '}_{2+\alpha}<\tilde{C}$. Now using the equi-boundedness and equi-continuity of a general subsequence $(w_{n_j})_j$ and of its derivatives $(\partial_\eta w_{n_j})_j$ and $(\partial_\zeta w_{n_j})_j$, by Ascoli-Arzela Theorem one can find subsequences $(w_{n_{j_k}})_k$, $(\partial_\eta w_{n_{j_k}})_k$ and $(\partial_\zeta w_{n_{j_k}})_k$ that converge uniformly on $\mcR'$ to continuous functions $w$, $\partial_\eta w$ and $\partial_\zeta w$ respectively. Being the uniform limit of the original sequence $w_n$ a constant equal to $v(t_0,x_0,0)$ then $\partial_\eta w_{n}$ and $\partial_\zeta w_{n}$ have to converge uniformly to $\partial_\eta w=0$ and $\partial_\zeta w=0$. So
  $$
  0\xleftarrow[\infty\leftarrow n]{|\cdot|^{\mcR'}_0} \zeta\partial_\eta w_{n}(s,\eta,\zeta) = \frac{\zeta}{m_n}\partial_y v\Big(\frac{s}{m_n}+t_n, \frac{\eta}{m_n}+x_n, \frac{\zeta}{m_n}\Big)
  $$
 and 
  $$
  0\xleftarrow[\infty\leftarrow n]{|\cdot|^{\mcR'}_0} \zeta\partial_\zeta w_{n}(s,\eta,\zeta) = \frac{\zeta}{m_n}\partial_x v\Big(\frac{s}{m_n}+t_n, \frac{\eta}{m_n}+x_n, \frac{\zeta}{m_n}\Big),
  $$
so, in particular, the limits hold if we restrict to the sequence $(t_n,x_n,y_n)_{n\in\N^*}$. Being $\partial_y v = \partial^2_y u$ and $\partial_x v = \partial_x \partial_y u$, then the proof is completed.
\end{proof}

\begin{remark} \label{remark_PDE_on_boundary}
 Under the hypotheses of Proposition \ref{min_hp_pde_existence_sol}, the equation $$
  \partial_t u(t,x,0) + (c \partial_x + a\partial_y +\varrho) u(t,x,0) =h(t,x,0)
  $$ 
 is satisfied for all $t\in[0,T)$, and so does not hold only as a limit. 
\end{remark}
\begin{proof}
 Let $t\in(0,T)$, since $\partial_y u$ is continuous up to the boundary, we expand $u$ in the direction $y$ around $(t,x,0)$ and $(t+\epsilon,x,0)$, and use the mean value theorem to get
  \begin{align*}
    \partial_t u(t,x,0)= \lim_{\epsilon \rightarrow 0^+} \frac{u(t+\epsilon,x,0) - u(t,x,0)}{\epsilon} &= \lim_{\epsilon\rightarrow 0^+}\frac{u(t+\epsilon,x,\epsilon^2) - u(t,x,\epsilon^2)+O(\epsilon^2)}{\epsilon} \\
    &= \lim_{\epsilon\rightarrow 0^+}  \partial_t u(t+\xi_\epsilon, x,\epsilon^2) +O(\epsilon),
  \end{align*}
 for some $\xi_\epsilon\in(0,\epsilon)$. Since for $y>0$ one has $ \partial_t u= -(\mcL +\varrho)u + h$, then by \eqref{0lims_dyyu_dxyu} and continuity of $\partial_x u$ and $\partial_y u$ up to the boundary $\{y=0\}$, one has 
  $$ 
  \partial_t u(t,x,0)=\lim_{\epsilon\rightarrow 0^+}-(\mcL+\varrho)u(t+\xi_\epsilon, x,\epsilon^2) + h(t+\xi_\epsilon, x,\epsilon^2) =  -(c \partial_x + a\partial_y +\varrho) u(t,x,0) +h(t,x,0).
  $$
Since the functions on both sides are continuous up to $t=0$, this is verified for $t=0$, too.
\end{proof}

So, we have just proven the following result.
\begin{prop}\label{hp_minimal_u_sol}
 Let $u$ defined as in \eqref{u_gen_sol}. Let $f$ and $h$ such that, as $j=0,1$, $\partial_x^{2j}f\in \mcC^{1-j}_{\pol}(\R\times\R_+)$,  $\partial_x^{2j}h\in \mcC^{1-j}_{\pol,T}(\R\times\R_+)$, $h$, $\partial_y h$ and $|h|^K_{\alpha,2}$, $|\partial_y h|^K_{\alpha,2}<\infty$ for all $K$ compact set contained in $[0,T)\times\R\times\R_+$.
 Then $\partial_x^{2j}u\in \mcC^{1-j}_{\pol,T}(\R\times\R_+)$ for $j=0,1$ and solves
  \begin{equation}
    \begin{cases}
    \partial_t u(t,x,y)+ \mcL u(t,x,y) +\varrho u(t,x,y)= h(t,x,y),\qquad& t\in [0,T), \,(x,y)\in \R\times\R_+,\\
 u(T,x,y)= f(x,y),\qquad& (x,y)\in \R\times\R_+.
    \end{cases} 
  \end{equation}
\end{prop}

We can compare the result we have just proven with the one in Proposition \ref{prop-reg-new}.
In the latter one, Briani et al. use the stochastic representation \eqref{stoc_repr-new_chap4}, with $q=2$, to prove the function $u$ belongs to $\mcC^{1,2}([0,T]\times\R\times\R_+)$ and so by continuity of all the derivatives involved in the problem,  that initially (for example taking only the final data $f$ just continuous) is solved only over $[0,T)\times\R\times\R^*_+$, is solved even over $[0,T)\times\R\times\{0\}$. 
To get all this regularity for $u$ one has to request a lot of regularity on $f$ and $h$: $\partial_x ^{2j}f\in \mcC^{2-j}_{\pol}(\R\times\R_+)$, $\partial_x ^{2j}h\in \mcC^{2-j}_{\pol, T}(\R\times\R_+)$, and for all $(m,n)\in\N^2$ such that $m+2n\le 4$ and $m\le 6$, $\partial^m_x\partial^n_y h$ locally Hölder continuous in $[0,T)\times\R\times\R_+^*$. We are capable of lowering these requests on $f$ and $h$ because even if  $u$ is less regular, it can still solve the reference problem. With the hypotheses considered in \eqref{hp_minimal_u_sol}, we do not have the continuity of $\partial_x\partial_y u$ and $\partial_y^2$, but we prove only \eqref{0lims_dyyu_dxyu} and this is enough, as shown in Remark \ref{remark_PDE_on_boundary}, to prove the PDE is solved over the boundary $[0,T)\times\R\times\{0\}$.

We now state sufficient conditions to ensure the uniqueness of the solution.
\begin{prop}
 There is at most one classical solution $u\in \mcC^{1,2}([0,T)\times(\R\times\R^*_+)) \cap \mcC^{1,1,1}([0,T)\times\R\times\R_+)\cap \mcC([0,T]\times\R\times\R_+)$ to the PDE \eqref{reference_PDE} such that the solution has polynomial growth in $(x,y)$ uniformly in $t$. In particular, under the hypothesis of Proposition \ref{hp_minimal_u_sol}, $u$ defined as in \eqref{u_gen_sol} is the unique solution.
\end{prop}
\begin{proof}
 Suppose that $u$ and $v$ are two solutions in the reference space $\mcC^{1,2}([0,T)\times(\R\times\R^*_+)) \cap \mcC^{1,1,1}([0,T)\times\R\times\R_+)\cap \mcC([0,T]\times\R\times\R_+)$, clearly the difference $w=u-v$ lies in the same space. For simplicity, we reverse the time by the change of variable $t\mapsto T-t$, so $w$ solves
  \begin{equation}
    \begin{cases}
 (\partial_t-\mcL -\varrho)w(t,x,y) = 0,\quad   t\in(0,T], &x\in\R, y\in\R_+, \\
 w(0,x,y) = 0, \quad  &x\in\R, y\in \R_+.
    \end{cases}
  \end{equation}
 From now on we consider $\varrho=0$, because $\exp(-\varrho t)w(t,x,y)$ solves the problem with the constant equal to 0, and the function $\mcM_{\varrho}:w(t,x,y)\mapsto \exp(-\varrho t)w(t,x,y)$ is a bijection from the reference space to itself.
 Let $L$ the first even integer such that $|w(t,x,y)|\le C(1+x^L+y^L)$. We call $h(x,y)=1+x^{L+2}+y^{L+2}$ and with a little algebra one can show that exists $M>0$ such that
  $$
 \mcL h(x,y)= \frac{(L+2)(L+1)}{2}y(\lambda^2 x^L+\sigma^2y^L)+(L+2)((c+d y)x^{L+1}+(a-by)y^{L+1})<Mh(x,y).
  $$
 Let $\varepsilon>0$, we define $w^\varepsilon:[0,T]\times\R\times\R_+\rightarrow \R$ by 
  $$w^\varepsilon(t,x,y) = w(t,x,y)+\varepsilon e^{Mt}h(x,y),$$ 
 then
  $$
 (\partial_t -\mcL) w^\varepsilon(t,x,y) = \varepsilon e^{Mt}(M-\mcL)h(x,y)>0
  $$
 for all the interior points. Let $\Gamma:=\{(t,x,y)\mid w^\varepsilon(t,x,y)< 0\} $, we remark that $\Gamma$ is bounded (because of the growth bound $|w(t,x,y)|\le C(1+x^L+y^L)$), and then $\bbar{\Gamma}$ is compact by continuity of $w^\varepsilon$. Assume that $\Gamma\neq \emptyset$ and define $t_0=\inf\{t\ge0\mid (t,x,y)\in\bbar{\Gamma} \text{, for some } (x,y)\in\R\times\R_+ \}$. We consider a point $(t_0,x_0,y_0)\in\bbar{\Gamma}$. By continuity of $w^\varepsilon$ and definition of $t_0$, $w$ must be equal 0 in $(t_0,x_0,y_0)$. In the meantime, being $w^\varepsilon(0,x,y)\ge 1$ and $\bbar{\Gamma}$ compact, one has $t_0>0$.
 We suppose first $y_0=0$. Then, by the fact that $t_0$ is an infimum, we have 
  $$
  \partial_t w^\varepsilon(t_0,x_0,0)\le 0, \qquad \partial_x w^\varepsilon(t_0,x_0,0)=0, \qquad\partial_y  w^\varepsilon(t_0,x_0,0)\ge 0 
  $$
 otherwise we can find a triple $(t_1,x_1,y_1)$ with $t_1<t_2$ such that $w^\varepsilon(t_1,x_1,y_1)<0$ contradicting the minimality of $t_0$. Being $a>0$, one has
  $$
  0\ge \partial_t w^\varepsilon(t_0,x_0,0) - a \partial_y w^\varepsilon(t_0,x_0,0) = \varepsilon e^{M t}M(1+x_0^{L+2})>0
  $$
 so $y_0=0$ is not possible. We consider now $y_0>0$, then $(t_0,x_0,y_0)$ is an interior point of the domain. By minimality of $t_0$ one has $\partial_t w^\varepsilon(t_0,x_0,y_0)\le 0$ and $(x_0,y_0)$ is a minimum point for the map $(x,y) \mapsto w^\varepsilon(t_0,x,y)$. Then, this map has a gradient equal to 0 and Hessian positive semi-definite, so $\mcL w^\varepsilon \ge 0$. One has
  $$
    0 \ge (\partial_t -\mcL) w^\varepsilon(t_0,x_0,y_0) = \varepsilon e^{Mt}(M-\mcL)h(x_0,y_0)>0.
  $$
 This contradiction implies that $\Gamma = \emptyset$. Since this holds for all $\varepsilon>0$, it follows that $w\ge0$. Applying the same argument to $-w$ shows that $w\le0$ and so $w=0$.
\end{proof}

\section{Existence and uniqueness of viscosity solutions}\label{viscosity_section}
Here, we want to explore the extended problem \eqref{reference_PDE} from the point of view of viscosity solutions in order to reduce the regularity on the function $f$. Now, we introduce some definitions (cf. \cite{CIL92}) that will be useful from now on.

Let $F:(0,T]\times\R^m\times\R\times\R^m\times\mcS(m)\rightarrow\R$ a continuous function where $\mcS(m)$ is the set of $m\times m$-dimensional, $\R$-valued symmetric matrices.
\begin{definition}[Degenerate ellipticity]
  $F$ is called degenerate elliptic if it is nonincreasing in its matrix argument
    $$
      F(t,x,u,p,X) \leq F(t,x,u,p,Y) \text{ for } Y\leq X,
    $$
    with the classical ordering $\leq$ over $\mcS(m)$ defined by the relation
    $$
      Y\leq X  \Leftrightarrow \langle Y\zeta,\zeta  \rangle \leq \langle X\zeta,\zeta  \rangle \text{ for all } \zeta\in\R^m.
    $$
\end{definition}

\begin{definition}[Proper]
  $F$ is called proper if it is degenerate elliptic and nondecreasing in $u$.
\end{definition}

\begin{remark}
  With the change of variable $s=T-t$ the original problem \eqref{reference_PDE} becomes
  \begin{equation}\label{reference_PDE_forward}
    \begin{cases}
      \partial_s u(s,x,y) +F(s,(x,y), u(s,x,y), D_{(x,y)}u(s,x,y),D^2_{(x,y)}u(s,x,y) )  = 0,  \\ 
      \phantom{\partial_s u(s,x,y) +F(s,(x,y), u(s,x,y), D_{(x,y)}u(s,x,y),D^2_{(x,y)})}\forall s\in(0,T],\, x\in\R,\, y\in\R_+, 
      \\
      u(0,x,y) = f(x,y), \quad \forall x\in\R, y\in \R_+,
    \end{cases}
  \end{equation}
  where
  \begin{equation}\label{F_logheston}
    F(s,(x,y),r,p,X)= -\frac{y}{2}(\lambda^2X_{1,1} +2\rho\lambda\sigma X_{1,2} +\sigma^2X_{2,2}) -\mu_X(y) p_1 -\mu_Y(y) p_2 - \varrho r+h(T-s,x,y)
  \end{equation}
  is degenerate elliptic (and proper if $\varrho\le0$).
\end{remark}

We consider a convex domain (possibly closed) $\mcO\subseteq \R^m$, $T>0$ and we name $\mcO_T = (0,T]\times\mcO$.
We study the following partial differential equation problem
\begin{equation}
  \begin{cases}
    u_t(t,x) + F(t,x,u(t,x),D_xu(t,x),D^2_xu(t,x))=0 \quad&\text{ if } t\in(0,T],\text{ } x\in\mcO,\\
    u(0,x)= f(x),   &\qquad\qquad\qquad x\in\mcO.
  \end{cases}
\end{equation}
We define 
\begin{align*}
  \LSC(\mcO_T) &= \{f:\mcO_T\rightarrow\R \mid   \text{$f$ is lower semi-continuous at every } (t,x)\in \mcO_T \}, \\
  \USC(\mcO_T) &= \{f:\mcO_T\rightarrow\R \mid   \text{$f$ is upper semi-continuous at every } (t,x)\in \mcO_T  \},
\end{align*}
and we give the following definition.
\begin{definition}\label{viscosity_property}
  Given a function $u$ and $(t,x)\in(0,T]\times\mcO$, we say that at $(t,x)$
  $$
  \partial_t u(t,x) +F(t,x,u(t,x),D_xu(t,x),D^2_x u(t,x))\ge 0\,(\text{resp.}\le 0)\text{ in viscosity sense }
  $$   if, for each smooth function $\phi$ such that $u-\phi$ has a local minimum (resp. a local maximum) at $(t,x)$ 
  $$
  \partial_t \phi(t,x) + F(t,x, u(t,x),D_x\phi(t,x),D^2_x\phi(t,x))\ge 0\,(\text{resp.}\le 0) \text{ in classical sense} .
  $$ 
\end{definition}
We introduce the notion of semijets to give an equivalent definition that will be useful in the following.
\begin{definition}\label{semijets}
  Let $u:\mcO_T\rightarrow \R$, then its upper parabolic second order semijet $\mcP_\mcO^{2,+} u$ is defined by
  \begin{align*}
    \mcP_\mcO^{2,+}u:\mcO_T&\rightarrow \mathscr{P}(\R\times\R^m\times\mcS(m))\\
    (t,x)&\mapsto \mcP_\mcO^{2,+}u(t,x)
  \end{align*}
  where $(c,p,X)$ lies in the set $\mcP_\mcO^{2,+}u(t,x)$ if 
  \begin{equation}\label{upper_semijet}
    \begin{split}
      u(s,z) \le  u(&t,x) +b(s-t)+ \langle p,z-x\rangle + \frac{1}{2}\langle X(z-x),x-z\rangle \\
      &+o(|s-t|+|z-x|^2) \text{ as } \mcO_T\ni(s,z)\rightarrow (t,x),
    \end{split}
  \end{equation}
  and we define the lower parabolic second order semijet $\mcP_\mcO^{2,-} u:= -\mcP_\mcO^{2,+}(-u)$.
  We also define the closure of these set-valued mappings as 
  \begin{multline}
    \bmcP_\mcO^{2,+}u(t,x)=\{(c,p,X)\in\R\times\R^m\times\mcS(m)\mid \exists\big((t_n,x_n),c_n,p_n,X_n\big)\in\mcO_T\times\R\times\R^m\times\mcS(m)\text{ s.t.} \\
    (c_n,p_n,X_n)\in \bmcP_\mcO^{2,+}u(t_n,x_n)\text{ and } \big((t_n,x_n),u(t_n,x_n),c_n,p_n,X_n\big)\rightarrow \big((t,x),u(t,x),c_n,p_n,X_n\big) \},
  \end{multline}
  and $\bmcP_\mcO^{2,-}u$, closure of $\mcP_\mcO^{2,-}u$, is defined in the same way.
\end{definition}

We now give the definition of viscosity super and sub-solutions. 
\begin{definition}[Viscosity super-solution (sub-solution)]\label{def_viscosity_sol}
  Let $F$ be a proper operator and $T>0$. A function $u$ that is $\LSC(\mcO_T)$  (resp. $\USC(\mcO_T)$)  is called a viscosity super-solution (resp. sub-solution) with initial value $f$ if, 
  \begin{itemize}
    \item for any $(t,x)\in\mcO_T$, $\partial_t\phi(t,x) + F(t,x,u(t,x),D_x u(t,x),D^2_x u(t,x))\geq 0 \text{ $($resp. $\le0)$}$ in viscosity sense,
    \item for any $x\in\mcO$, $u(0,x)\geq f(x) \text{ $($resp. $\le f(x)).$}$
  \end{itemize}
  A function $u$ that is both a super and a sub-solution is called a viscosity solution.
\end{definition}
 The following result gives an interesting equivalent definition.
\begin{prop}
  Let $(t,x)\in\mcO_T$. Then
  \begin{align}
    \mcP_\mcO^{2,+}u(t,x) =\{(\partial_t\phi(t,x),D_x\phi(t,x),D^2_x\phi(t,x))\mid \phi\text{ is } \mcC^{1,2} \text{ and } u-\phi  \text{ has local max. at } (t,x) \},\\
    \mcP_\mcO^{2,-}u(t,x) =\{(\partial_t\phi(t,x),D_x\phi(t,x),D^2_x\phi(t,x))\mid \phi\text{ is } \mcC^{1,2} \text{ and } u-\phi  \text{ has local min. at } (t,x) \}.
  \end{align}
\end{prop}
\begin{proof}
  The inclusion $\supseteq$ follows easily by using the local maximum (respectively minimum) property and developing $\phi$ using Taylor Theorem around $(t,x)$ up to the first order in $t$, and to the second in $x$. The nontrivial inclusion $\subseteq$ requires constructing, for every $(c,p,X)\in\mcP_\mcO^{2,+}u(t,x)$, a regular function such that the difference $u-\phi$ has a local minimum at $(t,x)$. We refer to Fleming and Soner \cite{FS93book}, V.4 Proposition 4.1.
\end{proof}
As an immediate consequence, we have the following characterization of super and sub-viscosity solutions.
\begin{corollary}
  A function $w\in\USC(\mcO_T)$ is a viscosity sub-solution with initial value $f$, if and only if
  \begin{equation}\label{subsol_2nd_def}
    \begin{cases}
      c + F(t,x,w(t,x),p,X)\leq 0 \text{ for all } (t,x)\in\mcO_T \text{ and } (c,p,X)\in\mcP_\mcO^{2,+}w(t,x), \\
      w(0,x)\leq f(x) , \text{ for any } x\in\mcO.
    \end{cases}
  \end{equation}
  A function $v\in\LSC(\mcO_T)$  is a viscosity super-solution with initial value $f$, if and only if
  \begin{equation}\label{supersol_2nd_def}
    \begin{cases}
      c + F(t,x,v(t,x),p,X)\geq 0,\text{ for all } (t,x)\in\mcO_T \text{ and } (c,p,X)\in\mcP_\mcO^{2,-}v(t,x), \\
      v(0,x)\geq f(x) , \text{ for any } x\in\mcO.
    \end{cases}
  \end{equation}
\end{corollary}

Here, we give a key Lemma to prove the verification Theorem, which says that viscosity solutions are stable under local uniform convergence. 
\begin{lemma}[Stability]\label{sequence_lemma_general}
  Let $F,(F_n)_{n\in\N}$ be continuous and degenerate elliptic such that for all $\mcK^*_T\subseteq \mcO_T\times\R\times\R^m\times\mcS(m)$
  $$
    |F-F_n|_0^{\mcK^*_T}\xrightarrow[n\rightarrow\infty]{}0,
  $$
  and $(u_n)_{n_\in\N}\subset C(\mcO_T)$ such that
  \begin{enumerate}
    \item for all $n$,  $\partial_t u_n(t,x) + F_n(t, x,u_n(t, x),D_x u_n(t, x),D^2_x u_n(t, x))\geq 0 $ $($resp. $\le0)$ for all $(t,x)\in\mcO_T$ in viscosity sense, \label{property_one _seq_lem_gen}
    \item there exists $u$ such that for each $\mcK_T\subseteq \mcO_T$ compact set one has 
    $$|u_n-u|_0^{\mcK_T}\xrightarrow[n\rightarrow\infty]{} 0.$$
  \end{enumerate}
  Then $u\in C(\mcO_T)$ satisfies $\partial_t u(t,x) + F(t, x,u(t, x),D_x u(t, x),D^2_x u(t, x))\geq 0 $ $($resp. $\le0)$  for all $(t,x)\in\mcO_T$ in viscosity sense.
\end{lemma}

\begin{proof}
  We only prove that $u$ satisfies $\partial_t u(t,x) + F(t, x,u(t, x),D_x u(t, x),D^2_x u(t, x))\geq 0$ in viscosity sense, the reverse inequality can be proven in the same way. Let $\phi\in \mcC^{1,2}(\mcO_T)$ and $(t,x)\in{\mcO_T}$ such that is a global minimum for $u-\phi$. We consider a compact neighbourhood $\mcK_T$ of $(t,x)$. Suppose the minimum is strict at $(t,x)$, then thanks to the local uniform convergence of $(u_n)_{n\in\N^*}$  exists a sequence of points $(t_n,x_n)_{n\in\N^*}$ eventually in the interior of $\mcK_T$ that are minima for the sequence $(u_n-\phi)_{n\in\N^*}$ and such that $(t_n,x_n)\rightarrow (t,x)$.
  Being $(t_n,x_n)$ minimizer for $u_n-\phi$ with $\phi$ smooth, then by \textit{(1)}
  \begin{equation*}
    0\le \partial_t \phi(t_n,x_n) + F_n(t_n,x_n,u_n(t_n,x_n),D_x\phi(t_n,x_n),D^2_x\phi(t_n,x_n)).
  \end{equation*}
  By uniform convergence of $u_n$ through $u$ over $\mcK_T$, one has $u_n(t_n,x_n)\rightarrow u(t,x)$. Then thanks to continuity of $D\phi$ and $D^2\phi$ one has that $(t_n,x_n,u_n(t_n,x_n),D_x u_n(t_n,x_n),D^2_x u_n(t_n,x_n) )_{n\in\N}\subset \mcK_T \times\mcR$ compact set of $(0,T]\times\mcO\times\R\times\R^m\times\mcS(d)$.
  So, thanks to uniformly convergence of $F_n$ through $F$  over $\mcK_T\times\mcR$  we conclude that
  \begin{multline}
    0\le \lim_{n\rightarrow\infty} \partial_t \phi(t_n,x_n) + F_n(t_n, x_n, u_n(t_n,x_n),D_x\phi(t_n,x_n),D^2_x\phi(t_n,x_n)) \\
    = \partial_t \phi(t,x) + F(t,x, u(t,x),D_x\phi(t,x),D^2_x\phi(t,x)).
  \end{multline}
 If the minimum is not strict we consider the function $\phi^*(s,y) = \phi(s,y)-(t-s)^2 -|x-y|^4$. This function $\phi^*$ is such that $u-\phi^*$ has strict min in $(t,x)$ and has same derivatives up to first order in $t$ and up to the second one in $(t,x)$, so we can apply the previous technique and conclude remarking
 \begin{multline*} 
  \lim_{n\rightarrow\infty} \partial_t \phi^*(t_n,x_n) + F_n(t_n,x_n, u_n(t_n,x_n), D_x\phi^*(t_n,x_n), D^2_x\phi^*(t_n,x_n))\\
  = \lim_{n\rightarrow\infty} \partial_t \phi(t_n,x_n) + F_n(t_n,x_n, u_n(t_n,x_n), D_x\phi(t_n,x_n), D^2_x\phi(t_n,x_n)). 
\end{multline*}
\end{proof}

We want to use the notion of viscosity solution to extend the verification results obtained in the previous section. Given 
\begin{equation}\label{u_sol_chap3}
  u(t,x,y)=\E\bigg[e^{\varrho (T-t)}f(X^{t,x,y}_T,Y^{t,y}_T)-\int_t^Te^{\varrho (s-t)} h(s,X^{t,x,y}_s,Y^{t,y}_s)ds\bigg],
\end{equation}
we will verify that $u^F$ (candidate forward solution)
\begin{equation}\label{uF_sol_chap3}
  u^F(t,x,y)=u(T-t,x,y)
\end{equation}
is a viscosity solution of the forward problem \eqref{reference_PDE_forward} with more general initial data $f$ (even with discontinuities) and source terms $h$. We emphasize that whenever we have $u$ solution of \eqref{reference_PDE}, we know that $u^F$ is a solution of \eqref{reference_PDE_forward} and vice versa.

\subsection{Continuous initial data}
In this subsection, we deal with problem \eqref{reference_PDE_forward} where the initial data $f$ is chosen to be just continuous.

In the sequel, we will need some smoothing arguments, which can be resumed as follows.
\begin{lemma}\label{Mollifier_lemma}
  Let $f\in \mcC(\R\times\R_+)$ and $h\in \mcC([0,T]\times \R\times\R_+)$.
  Then there exist 
  \begin{itemize}
    \item $(f_n)_{n\in\N}\subset \mcC^ \infty(\R^2) $ such that $|f_n-f|_0^{\mcK}\rightarrow 0$ for every compact set $\mcK \subset \R\times\R_+$,
    \item $(h_n)_{n\in\N}\subset \mcC^ \infty(\R^3) $ such that $|h_n-h|_0^{\mcK_T}\rightarrow 0$ for every compact set $\mcK_T \subset [0,T]\times\R\times\R_+$.
  \end{itemize}
  Furthermore, if $f$ and $h$ are uniformly continuous and bounded then $|f_n-f|_0^{\R\times\R_+}\rightarrow 0$ and $|h_n-h|_0^{[0,T]\times\R\times\R_+}\rightarrow 0$.
\end{lemma}
\begin{proof}
  We need only to extend $f$ and $h$ in a continous way respectively around $\R\times\R_+$ and $[0,T]\times \R\times\R_+$ and than to take convolution with a mollifier $(\varphi_n)_{n\in\N}$. We finish applying Proposition 4.21 of \cite{Brezis2010book}.  
  We start by extending $f$ and $h$ in the following continuous way
  \begin{equation*}
    \tf(x_1,x_2)= f(x_1,0\vee x_2) \quad{ and } \quad \tilde{h}(t,x_1,x_2) = h(0\vee t\wedge T,x_1,x_2).
  \end{equation*}
  If $f$ and $h$ are uniformly continuous and bounded, then the same proof in \cite{Brezis2010book} guarantees uniform convergence without any restriction over compact sets (so in $\R\times\R_+$ and $[0,T]\times\R\times\R_+)$.
\end{proof}

We recall that for every compact set $\mcK_T$ in $[0,T]\times\R\times\R_+$ and $p\in\N$, one has
  \begin{equation}
    \sup_{(t,x,y)\in\mcK_T} \E\left[\big|X_T^{t,x,y}\big|^p+\big|Y^{t,y}_T\big|^p\right] < \infty,
  \end{equation}
where $\big((X_T^{t,x,y},Y^{t,y}_T)\big)_{t\in[0,T]}$ denotes the solution to \eqref{referenceDiffusion}.

We are now ready to prove that \eqref{reference_PDE} has a viscosity solution, with quite general requests on the function $f$ giving the Cauchy condition. We underline that we do not operate any restriction on the parameters: no Feller's condition is required.

\begin{prop}\label{existence_viscsol_fcont}
  Let $f\in C(\R\times \R_+)$ and $h\in C([0,T)\times\R\times\R_+)$ be such that for all compact set $\mcK_T\in[0,T]\times\R\times\R_+$ there exists $p>1$ such that
  \begin{equation}\label{integrability_condition}
    \sup_{(t,x,y)\in\mcK_T}\|f(X_T^{t,x,y},Y^{t,y}_T)\big\|_{L^p(\Omega)}, \sup_{(t,x,y)\in\mcK_T}\int_t^T \|h(s,X^{t,x,y}_s,Y^{t,y}_s)\|_{L^p(\Omega)}ds<\infty .
  \end{equation}
  Then, $$u(t,x,y)=\E\left[e^{\varrho (T-t)}f(X_T^{t,x,y},Y^{t,y}_T)-\int_t^T e^{\varrho (s-t)} h(s,X^{t,x,y}_s,Y^{t,y}_s)ds\right]$$ is $C([0,T]\times\R\times\R_+)$ and is a viscosity solution to the PDE \eqref{reference_PDE}.
\end{prop}

\begin{proof}
  In order to simplify the proof, we write $u=v-w$ where
  $$
    v(t,x,y) = \E[e^{\varrho (T-t)}f(X_T^{t,x,y},Y^{t,y}_T)] \quad\text{ and }\quad w(t,x,y) = \E\left[\int_t^T e^{\varrho (s-t)} h(s,X^{t,x,y}_s,Y^{t,y}_s)ds\right],
  $$
  and we will show all the due convergences in the ``$w$'' part, being the $v$ part similar and simpler.
  Let $R>1/T$ and consider the smoothly truncated version $f^R$ and $h^R$ of $f$ and $h$, that is, $f^R(x,y)=f(x,y)\zeta_R(x,y)$ and $h^R(t,x,y)=h(t,x,y)\xi_R(t)\zeta_R(x,y)$, where  $\zeta_R$ and $\xi_R$ are smooth function such that 
  $$
  \ind{B(0,R)}\leq \zeta_R\leq \ind{B(0,R+1)}\quad \text{ and }\quad\ind{[0,T-\frac{1}{R}]}\leq \xi_R\leq \ind{[0,T-\frac{1}{R+1}]}.
  $$
  $f^R$ and $h^R$, being continuous and having compact support, are, in particular, uniformly continuous and bounded. Then by Lemma \ref{Mollifier_lemma} there exist two sequences $(f^R_n)_{n\in\N}$, $(h^R_n)_{n\in\N}$ that approximate in uniform norm $f^R$ over $\R\times\R_+$  and $h^R$ over $[0,T]\times\R\times\R_+$.
  We define in an intuitive way $v^R,v^R_n, w^R \text{ and } w^R_n$ as the functions obtained by replacing in $v$ and $w$ the functions $f$ and $h$ with $f^R$, $f^R_n$ and $h^R$, $h^R_n$. 
  Being the sequences $(f^R_n)_{n\in\N}\subset C^\infty_c(\R\times\R_+)\subset \mcC^{\infty}_{\pol}(\R\times\R_+)$ and $(h^R_n)_{n\in\N}\subset C^\infty_c([0,T]\times\R\times\R_+)\subset \mcC^{\infty}_{\pol}([0,T]\times\R\times\R_+)$, they satisfy the regularity conditions in Proposition \ref{prop-reg-new}, then for all $n\in\N$, $u^R_n(t,x,y)$ is a $\mcC^\infty$ classical solution to \eqref{reference_PDE} with final value $f^R_n$ and source term $h^R_n$ instead of $h$. Now, chosen a compact set $\mcK_T\subset[0,T]\times\R\times\R_+$, for all $R>1/T$ one has
  \begin{align*}
    |u^R-u^R_n|_0^{\mcK_T} \le& \, |v^R-v^R_n|_0^{\mcK_T} +|w^R-w^R_n|_0^{\mcK_T} \\
    =& \sup_{(t,x,y)\in \mcK_T} \big|\E\big[e^{\varrho (T-t)}\big(f^R(X_T^{t,x,y},Y^{t,y}_T)-f^R_n(X_T^{t,x,y},Y^{t,y}_T)\big)\big]\big| \\
    &+ \sup_{(t,x,y)\in \mcK_T} \bigg|\E\Big[\int_t^T e^{\varrho (s-t)} \big(h^R(s,X^{t,x,y}_s,Y^{t,y}_s)-h^R_n(s,X^{t,x,y}_s,Y^{t,y}_s)\big)ds\Big]\bigg| \\
    \le& \, C_1|f^R-f^R_n|_0^{\R\times\R_+}+ C_2|h^R-h^R_n|_0^{[0,T]\times\R\times\R_+}\longrightarrow 0.
  \end{align*}
  We show now that $|w^R-w|_0^{\mcK_T}\rightarrow0$ when $R\rightarrow\infty$, similarly one can do the same for $v$ and get the convergence in uniform norm for $u$. We write $z=(x,y)$ and $ Z^{t,z}_T = (X_T^{t,x,y},Y^{t,y}_T)$ and show that 
  \begin{align*}
    \bigg|\E\Big[\int_t^T e^{\varrho (s-t)} \big(h(s,X^{t,x,y}_s,Y^{t,y}_s)-&h^R(s,X^{t,x,y}_s,Y^{t,y}_s)\big)ds\Big]\bigg|\\
    &\le 2e^{(0\vee\varrho)T} \E\Big[\int_t^T |h(s,Z^{t,z}_s)|\ind{([0,T-\frac{1}{R}]\times B_R(0))^C}(s,Z^{t,z}_s)ds\Big],   \\
    &\le 2e^{(0\vee\varrho)T} \int_t^T\E\big[ |h(s,Z^{t,z}_s)| \ind{B^C_R(0)}(Z^{t,z}_s)\big] ds,   \\
    &\le 2e^{(0\vee\varrho)T} \int_t^T \|h(s,Z^{t,z}_s)\|_{L^P(\Omega)}ds\, \P(|Z^{t,z}_T|>R)^{\frac{p-1}{p}}, 
  \end{align*}
  where we used the Tonelli Theorem to exchange the order of the expected value and the integral and the Hölder inequality to get the last inequality. Now, using Markov inequality,
  \begin{equation*}
    \P(|Z^{t,z}_T|>R)\leq \frac{\E[|Z^{t,z}_T|]}{R}.
  \end{equation*}
  Using this last inequality and passing to the supremum over $\mcK_T$, we get
  {\small
   \begin{equation*}
    |w^R-w|_0^{\mcK_T}\le C \sup_{(t,x,y)\in\mcK_T}\int_t^T \|h(s,X^{t,x,y}_s,Y^{t,y}_s)\|_{L^p(\Omega)}ds \frac{\sup_{(t,x,y)\in\mcK_T}\E[|(X_T^{t,x,y},Y^{t,y}_T)|]^{\frac{p-1}{p}}}{R^\frac{(p-1)}{p}}\xrightarrow[R\rightarrow\infty]{} 0,
   \end{equation*}}
   that proves the limit. 
  Furthermore,  thanks to triangular inequality, one has 
  $$
  |u-u^R_n|_0^{\mcK_T} \le  |u-u^R|_0^{\mcK_T} + |u^R-u^R_n|_0^{\mcK_T},
  $$
  so, for all $n\in\N^*$, we can pick $R_n$ such that the second norm on the right-hand side is upper bounded by $1/(2n)$. Then, replacing $R$ with $R_n$ in the first norm of the right-hand part of the inequality, it exists $k(n)$ such that this norm is upper bounded by $1/(2n)$ too. So we define $\hat{u}_n = u^{R_n}_{k(n)}$ and
  \begin{equation}\label{convergence_prop_existence_f_cont}
    |u-\hat{u}_n|_0^{\mcK_T}\le \frac{1}{n}.
  \end{equation} 
  The functions $(u^F_n)_{n\in\N}=\big(\hat{u}_n(T-\cdot,\cdot,\cdot)\big)_{n\in\N}$ are classical solutions (so in particular viscosity solutions) to \eqref{reference_PDE_forward} where we have $F_n$ instead of $F$ by replacing in it $h$ with $h^{R_n}_{k(n)}$. $u$ (and so $u^F$) is $C([0,T]\times\R\times\R_+)$ thanks to \eqref{convergence_prop_existence_f_cont}. Furthermore is easy to prove the uniform convergence hypothesis $F_n\rightarrow F$ (because $h^{R_n}_{k(n)}\rightarrow h$ uniformly over the compact sets of $[0,T)\times\R\times\R_+$), so thanks to Lemma \ref{sequence_lemma_general}, $u^F$ satisfies $$\partial_t u^F(t,x) + F(t, x,u^F(t, x),D_x u^F(t, x),D^2_x u^F(t, x))= 0 \text{ for } (t,x,y)\in(0,T]\times\R\times\R_+$$ in viscosity sense. Furthermore, $u^F(0,x,y)=f(x,y)$ for every $(x,y)\in\R\times\R_+$, so $u^F$ is a viscosity solution of \eqref{reference_PDE_forward} with initial value $f$.
\end{proof}

\begin{remark}
  The hypothesis \eqref{integrability_condition} with $p>1$ is not restrictive. For example all the functions $f\in C_{\pol}(\R\times\R_+)$, $h\in C_{\pol,T}(\R\times\R_+)$ satisfy this hypothesis for all $p>1$. In fact
  $$
  \sup_{(t,x,y)\in\mcK_T}\|f(X_T^{t,x,y},Y^{t,y}_T)\|_{L^p(\Omega)} \le [C(1 + \sup_{(t,x,y)\in\mcK_T}\E[|(X_T^{t,x,y}|^{ap}+|Y^{t,y}_T)|^{ap}])]^{\frac{1}{p}}< \infty,
  $$
  and
  $$
  \sup_{(t,x,y)\in\mcK_T}\int_t^T \|h(s,X^{t,x,y}_s,Y^{t,y}_s)\|_{L^p(\Omega)}ds \le C(1 + \sup_{(t,x,y)\in\hat{\mcK}_T}\E[|(X_T^{t,x,y}|^{ap}+|Y^{t,y}_T)|^{ap}])< \infty
  $$
  with $\hat{\mcK}_T=\{(t,x,y)\in[0,T]\times\R\times\R_+\mid \exists t_0\in[0,T] \text{ such that } (t_0,x,y)\in\mcK_T \}$.
\end{remark}

\subsection{Comparison principle and uniqueness for continuous initial data}
In this subsection, we prove a comparison principle for our reference PDE \eqref{reference_PDE_forward}. In Subsection \ref{discont_initial_data_subsec}, we prove this result allows getting uniqueness of solutions even for initial data $f$ that present ``some discontinuities''.

We start by stating two results that will be crucial to prove a comparison argument needed to prove the uniqueness of the solution. We first recall a lemma proved in \cite{CIL92}, Proposition 3.7.

\begin{lemma}\label{lemma_doubling}
  Let $M\in\N^*$, $\mcA$ be a subset of $\R^M$, $\Phi\in \USC(\mcA)$, $\Psi\in \LSC(\mcA)$,
  \begin{equation}
    M_\alpha = \sup_{\mcA}(\Phi(x)-\alpha\Psi(x))
  \end{equation}
  for $\alpha>0$. Let $\lim_{\alpha\rightarrow\infty} M_\alpha \in\R$ and $x_\alpha\in\mcA$ be chosen so that 
  \begin{equation}
    \lim_{\alpha\rightarrow\infty} \big(M_\alpha-(\Phi(x_\alpha)-\alpha\Psi(x_\alpha))\big)=0.
  \end{equation}
  Then, the following hold
  \begin{equation}\label{implication_lemma_doubling}
    \begin{cases}
      (i)&\lim_{\alpha\rightarrow\infty} \alpha\Psi(\alpha)=0, \\
      (ii)& \Psi(\hx)=0 \text{ and } \lim_{\alpha\rightarrow\infty} =\Phi(\hx)=\sup_{\{x\in\mcA\mid\Psi(x)=0\}}\Phi(x)\\
       &\qquad\qquad\qquad\text{whenever } \hx\in\mcA \text{ is a limit point of } x_\alpha \text{ as } \alpha\rightarrow\infty.
    \end{cases}
  \end{equation}
\end{lemma}

\begin{prop}\label{semijet_estimates}
  Let $u_i\in\USC((0,T]\times\mcO_i)$ for $i=1,\ldots,k$ where $\mcO_i$ is s locally compact subset of $\R^{N_i}$. Let $\varphi$ be defined on an open neighborhood of $(0,T]\times\mcO_1\times\cdots\times\mcO_k$ and such that $\varphi:(t,x_1,\ldots,x_k)\mapsto \varphi(t,x_1,\ldots,x_k)$ is once continuously differentiable in $t$ and twice continuously differentiable in $(x_1,\ldots,x_k)\in\mcO_1\times\cdots\times\mcO_k$. Suppose that $\hat{t}\in(0,T]$, $\hx_i\in\mcO_i$ for $i=1,\ldots,k$ and
  \begin{align*}
    w(t,x_1,\ldots,x_k) &\equiv u_1(t,x_1) + \cdots + u_k(t,x_k) - \varphi(t,x_1,\ldots,x_k)\\
    &\le w(\hat{t},\hx_1,\ldots,\hx_k),
  \end{align*}
  for $0<t\le T$ and  $x_i\in\mcO_i$. Assume, moreover, that there is an $r>0$ such that for every $M>0$ there is a $C$ such that for $i=1,\ldots,k$
  \begin{multline}
    b_i\le C \text{ whenever } (b_i,q_i,X_i)\in\mcP^{2,+}_{\mcO_i}u_i(t,x_i),\\
    |x_i-\hx_i|+|t-\hat{t}|\le r \text{ and } |u_i(t,x_i)|+|q_i|+\|X_i\|\le M.
  \end{multline}
  Then for each $\varepsilon>0$ there are $X_i\in S(N_i)$ such that 
  \begin{equation}
    \begin{cases}
      (i)& (b_i,D_{x_i}\varphi(\hat{t},\hx_1,\ldots,\hx_k),X_i )\in\bmcP^{2,+}_{\mcO_i} u_i(\hat{t},\hx_i) \text{ for } i=1,\ldots,k, \\
      (ii)& -\left(\frac{1}{\varepsilon} +\|A\| \right) I \le \begin{pmatrix}
        X_{1} & \cdots  & 0 \\
        \vdots & \ddots & \vdots \\
        0& \cdots  & X_{k}  
    \end{pmatrix} \le A+\varepsilon A^2\\
      (iii)& b_1+\cdots+b_k = \partial_t\varphi(\hat{t},\hx_1,\ldots,\hx_k),
    \end{cases}
  \end{equation}
  where $A=(D^2_x\varphi)(\hat{t},\hx_1,\ldots,\hx_k)$.
\end{prop}

\begin{proof}
  We refer to the proof in \cite{CI90} with a small modification. Here, the $u_i$-s are $\USC$ up to $T$, so we must consider possible maximum points over ${T}\times\mcO_1\times\cdots\times\mcO_k$. In the proof \cite[Lemma 8]{CI90} we redifine the $v_i$-s functions equal to $-\infty$ only when $|x_i|>1$ and $t<s/2$, so  $t_{i,\delta}$ belongs to $[s/2,T]$ and the rest of the proof still the same.
\end{proof}
\
We can now state a comparison principle for semicontinuous functions of the problem \eqref{reference_PDE_forward}.

\begin{prop}\label{comparison_principle}
  \textbf{(Comparison principle)}
  Let $w\in\USC([0,T]\times\R\times\R_+)$ and  $v\in\LSC([0,T]\times\R\times\R_+)$ be respectively a subsolution and a supersolution to \eqref{reference_PDE_forward} with polynomial growth uniformly in time, where $F$ is as in \eqref{F_logheston}. Then $w\le v$.
\end{prop}
\begin{proof}
  It is simple to show (using the Definition \ref{def_viscosity_sol}) that if $w$ and $v$ are subsolution and supersolution to the general problem \eqref{reference_PDE_forward} with $\varrho\in\R$ then $e^{-\theta  t}w(t,x,y)$ and $e^{-\theta t}v(t,x,y)$ are subsolution and supersolution to \eqref{reference_PDE_forward} with $\varrho$ replaced by $\varrho-\theta$.
  So we need only to prove the result when $\varrho=0$.
  We start by remarking that $w$, $v$ and $h$ have polynomial growth uniformly in time, then
  \begin{equation}\label{pol_growth}
    \sup_{t\in[0,T]}|w(t,x,y)|,\sup_{t\in[0,T]}|v(t,x,y)|,\sup_{t\in[0,T]}|h(t,x,y)|\le C(1+x^L+y^L) \text{ for some } C>0,  L\in2\N^* 
  \end{equation} 
  
  We define the function $\phi_\varepsilon(t,x,y)=\varepsilon e^{Mt}(1+x^{L+2}+y^{L+2})$. It's easy to check that $\phi_\varepsilon$ is $C^\infty([0,T]\times\R\times\R_+)$ and, for every $\varepsilon$, $M>0$ can be chosen such that
  \begin{equation}\label{classical_strict_supersolution}
    \begin{cases}
      (\partial_t-\mcL)\phi_\varepsilon(t,x,y) \ge \varepsilon,\quad  \forall t\in(0,T],\, x\in\R,\, y\in\R_+, \\
      \phi_\varepsilon(0,x,y) \ge \varepsilon, \quad \forall x\in\R, y\in \R_+.
    \end{cases}
  \end{equation}
  Then, called $v_\varepsilon = v+\phi_\varepsilon$ and $\mcO=\R_+\times\R$, for all $(t,x,y)\in\mcO_T=(0,T]\times\mcO$ one has   
  \begin{multline*}
    \mcP^{2,-}_{\mcO_T}(v+\phi_\varepsilon)(t,x,y)= \big\{ (\alpha +\partial_t\phi_\varepsilon(t,x,y),(\beta,\gamma) +D_{(x,y)}\phi_\varepsilon(t,x,y), X + D^2_{(x,y)}\phi_\varepsilon(t,x,y))  \mid \\ (\alpha,(\beta,\gamma),X)\in\mcP^{2,-}_{\mcO_T}v(t,x,y) \big\},
  \end{multline*}
 and by the linearity of the reference PDE, one has that
 \begin{equation}\label{strict_supersolution}
  c + F(t,(x,y),v_\varepsilon(t,x,y),p,X)\geq \varepsilon,\text{ for all } (t,x,y)\in\mcO_T \text{ and } (c,p,X)\in\mcP_{\mcO_T}^{2,-}v_\varepsilon(t,x,y),
 \end{equation}
 and clearly $w-v_\varepsilon\le-\varepsilon$, that means $v_\varepsilon$ is a strict super-solution.
 Thanks to the growth hypothesis \eqref{pol_growth}
 \begin{equation*}
  w(t,x,y)-v_\varepsilon(t,x,y)\le 2C(1+x^L+y^L) -\varepsilon (1+x^{L+2}+y^{L+2}) \xrightarrow[|(x,y)|\rightarrow\infty]{} -\infty,
 \end{equation*}
 then there exists $R>0$ such that,  $w-v_\varepsilon\le0$ outside a rectangle $[0,T]\times\mcR$, where $\mcR=\times[-R,R]\times[0,R]$.
 We now suppose that there exist a point $(t_0,x_0,y_0)\in\mcR$ such that $(u-v_\varepsilon)(t_0,x_0,y_0)=\delta>0$. If such a point exists, $t_0$ must be $>0$ by the initial condition.
 In order to come up with a contradiction we use the well known technique in classical viscosity solutions framework of doubling the variables.
 We define, for all $\alpha>0$
 $$\varphi_\alpha:([0,T]\times\R^2\times\R^2)\ni (t,\eta,\zeta)\mapsto \frac{1}{2}|\eta-\zeta|^2\in\R_+.$$
 We penalize the function $w-v_\varepsilon$ subtracting the function $\alpha\varphi_\alpha$, while we double the spatial variables, 
 and study 
 $$M_\alpha=\sup_{(t,\eta,\zeta)\in[0,T]\times\mcR\times\mcR}w(t,\eta)-v_\varepsilon(t,\zeta) -\alpha\varphi_\alpha(t,\eta,\zeta). $$
 By the upper semi-continuity of the function $w(t,\eta)-v_\varepsilon(t,\zeta) -\alpha\varphi_\alpha(t,\eta,\zeta)$, and compactness of $[0,T]\times\mcR\times\mcR$, 
 $$\delta\le M_\a =  w(t_\a,\eta_\a)-v_\varepsilon(t_\a,\zeta_\a) -\alpha\varphi_\alpha(t_\a,\eta_\a,\zeta_\a)$$ for some $(t_\a,\eta_\a,\zeta_\a)\in[0,T]\times\mcR\times\mcR$. 
 We apply now Lemma \ref{lemma_doubling} with $\mcO= [0,T]\times\mcR\times\mcR$, $\Psi=\varphi$ and we chose the point $x_\alpha$ in the lemma to be the point $(t_\a,\eta_\a,\zeta_\a)$ that realizes the maximum $M_\alpha$. Then \eqref{implication_lemma_doubling} translates to
 \begin{equation}\label{implication_lemma_doubling_parab}
  \begin{cases}
    (i)&\lim_{\alpha\rightarrow\infty} \frac{\a}{2}|\eta_\a-\zeta_\a|^2=0, \\
    (ii)& \lim_{\alpha\rightarrow\infty} M_\a = u(\hat{t},\hat{\eta})-v_\varepsilon(\hat{t},\hat{\eta})=\sup_{\{(t,\eta)\in[0,T]\times\mcR\}}w(t,\eta)-v_\varepsilon(t,\eta)\\
     &\qquad\text{whenever } (\hat{t},\hat{\eta})\in[0,T]\times\mcR\ \text{ is a limit point of } (t_\a,\eta_\a) \text{ as } \alpha\rightarrow\infty.
  \end{cases}
 \end{equation}
 Now, because of the initial condition, \eqref{implication_lemma_doubling_parab} $(i)$ and $(ii)$, $(t_\a,\eta_\a,\zeta_\a)$ must lie inside $(0,T]\times\mcR\times\mcR$ for large $\a$.
 We want to show that there exists values in $\bmcP^{2,+}_{(0,T]\times\mcR}w(t_\a,\eta_\a)$ and $\bmcP^{2,-}_{(0,T]\times\mcR}v_\varepsilon(t_\a,\eta_\a)$ that are not compatible.
 We apply Proposition \ref{semijet_estimates} to the point $(t_\a,\eta_\a,\zeta_\a)$ with $u_1=w$ and $u_2=-v_\varepsilon$, $\varphi=\varphi_\a$, $\mcO_1=\mcO_2=\R\times\R_+$,  and $\varepsilon$ equal to $\alpha^{-1}$, one get
 $$
  (c_1,\alpha(\eta_\a-\zeta_\a),X_\alpha)\in\bmcP^{2,+}_\mcO w(t_\a,\eta_\a),\qquad (c_2,\alpha(\eta_\a-\zeta_\a),Y_\alpha)\in\bmcP^{2,-}_\mcO v_\varepsilon(t_\a,\zeta_\a)
 $$
 such that $c_1=c_2$ and
 \begin{equation*}
   -3\a\begin{pmatrix}
    I & 0 \\
    0 &  I
  \end{pmatrix} \le \begin{pmatrix}
    X_\a & 0 \\
    0 & -Y_\a
  \end{pmatrix} \le 3\a \begin{pmatrix}
    I  & -I \\
    -I &  I
  \end{pmatrix},
 \end{equation*}
 from which we derive
 \begin{equation}
  X_\a\le Y_\a \quad\text{ and }\quad |\langle X_\a z,z\rangle|\le 3 \alpha|z|^2.
 \end{equation}
 Using that $u$ is a sub-solution \eqref{subsol_2nd_def} and $v_\varepsilon$  is a strict super-solution \eqref{strict_supersolution}, one has
 \begin{align*}
  c_1 + F(t_\a,\eta_\a, w(t_\a,\eta_\a), \a(\eta_\a-\zeta_\a),X_\a)&\le 0, \\
  c_2 + F(t_\a,\zeta_\a, v_\varepsilon(t_\a,\zeta_\a), \a(\eta_\a-\zeta_\a),Y_\a)&\ge \varepsilon,
 \end{align*}
 using $X_\a\le Y_\alpha$, $v_\varepsilon(t_\a,\zeta_\a)\le w(t_\a,\eta_\a)$   and that $F$ is proper ($\varrho=0$) imply
 \begin{align*}
  \varepsilon &\le F(t_\a,\zeta_\a, v_\varepsilon(t_\a,\zeta_\a), \a(\eta_\a-\zeta_\a),Y_\a) -F(t_\a,\eta_\a,  w(t_\a,\eta_\a), \a(\eta_\a-\zeta_\a),X_\a) \\
  &\le F(t_\a,\zeta_\a, w(t_\a,\eta_\a), \a(\eta_\a-\zeta_\a),X_\a) - F(t_\a,\eta_\a,  w(t_\a,\eta_\a), \a(\eta_\a-\zeta_\a),X_\a).
 \end{align*}
 Naming $\eta_\a =(x^\eta_\a,y^\eta_\a)$, $\zeta_\a =(x^\zeta_\a,y^\zeta_\a)$, and using the definition of $F$ in \eqref{F_logheston} one has

 \begin{multline}
  \varepsilon\le \frac{(y^\zeta_\a-y^\eta_\a)}{2}(\lambda^2X^{1,1}_\a+2\rho\lambda\sigma X^{1,2}_\a+\sigma^2 X^{2,2}_\a)+d(y^\zeta_\a-y^\eta_\a)\alpha(x^\zeta_\a-x^\eta_\a) \\
  -b(y^\zeta_\a-y^\eta_\a)\alpha(y^\zeta_\a-y^\eta_\a) +h(T-t_\a,x^\zeta_\a,y^\zeta_\a)-h(T-t_\a,x^\eta_\a,y^\eta_\a).
 \end{multline}
 The first term on the right-hand side is upper bounded by $3/2(\lambda^2+\sigma^2)\alpha(y^\zeta_\a-y^\eta_\a)$, so the first three terms go to 0 thanks to \eqref{implication_lemma_doubling_parab} $(i)$ while $h(T-t_\a,x^\zeta_\a,y^\zeta_\a)-h(T-t_\a,x^\eta_\a,y^\eta_\a)$ goes to 0 when $\alpha\rightarrow\infty$ by continuity of $h$ (up to choose a subsequence of $\alpha$s for which $(t_\a,\eta_\a)\rightarrow(\hat{t},\hat{\eta})$). So $\varepsilon$, that is strictly positive, is upper bounded by a quantity that goes to 0 when $\alpha\rightarrow\infty$.
 This contradiction yields that there is no point $(t_0,x_0,y_0)$ in which $u-v_\varepsilon>0$ than $u\le v_\varepsilon$ for every $\varepsilon>0$ and so $u\le v$.
\end{proof}

We finish by stating and proving the two following results, which discuss the uniqueness of the reference PDE and the regularity of the solution.

\begin{corollary}
  The problem \eqref{reference_PDE} has a unique viscosity solution that is continuous over $[0,T]\times\R\times\R_+ $ and that has polynomial growth in $(x,y)$ uniformly in $t$. 
\end{corollary}

\begin{proof}
  For simplicity, we consider the equivalent forward PDE \eqref{reference_PDE_forward} and two solutions $u_1$ and $u_2$ with polynomial growth. Then, by the linearity of the PDE, $u=u_1-u_2$ is a viscosity solution of
  \begin{equation}
    \begin{cases}
      (\partial_t-\mcL)u(t,x,y) = 0,\quad   t\in(0,T], &x\in\R, y\in\R_+, \\
      u(0,x,y) = 0, \quad  &x\in\R, y\in \R_+,
    \end{cases}
  \end{equation}
  with polynomial growth, so it exists $L\in2\N$ such that $|u(t,x,y)|\le C(1+x^L+y^L)$.
  We remark that $v(t,x,y)=0$ is a solution to the problem, then by Proposition \ref{comparison_principle} $u\le v$ and $u\ge v$, so $u_1=u_2$ is the unique solution.

\end{proof}

\begin{prop}\label{uniqueness_and_regularity_f_cont}
  Let $f\in \mcC_{\pol}(\R\times\R_+)$ and $h\in \mcC_{\pol,T}(\R\times\R_+)$.
  Then $u$ as in \eqref{u_gen_sol} is the unique viscosity solution of \eqref{reference_PDE} that belongs to $\mcC([0,T]\times\R\times\R_+)$ and that has polynomial growth in $(x,y)$ uniformly in $t$. 
  
  Furthermore, if $h$ is locally Hölder in the compact sets of $[0,T)\times\R\times\R^*_+$ then $u\in \mcC^{1,2}\big([0,T)\times(\R\times\R^*_+)\big)$.
\end{prop}
\begin{proof}
  The first hypotheses over $f$ and $h$ guarantee $u$ to be continuous over $[0,T]\times\R\times\R_+ $ and to have polynomial growth, so to be the unique solution. Furthermore, if $f$ and $h$ are locally Hölder in the compact sets of $[0,T)\times\R\times\R^*_+$, one can consider the PDE locally in a compact inside $[0,T)\times\R\times\R^*_+$ to prove further regularity.
  Let $t<S\in[0,T)$, $(x,y)\in\R\times\R_+^*$ and $\mcR=(x-R,x+R) \times (y/2,2y)$, $R>0$,  $Q=[0,S)\times\mcR$ and consider the PDE problem
  \begin{equation}
  \begin{cases}\label{parabolic_problem_bounded_domain}
  \partial_tv+ \mcL v+\varrho v= h,\qquad&\mbox{ in }Q,\\
  v=u,\qquad&\mbox{ in }\partial_0Q,
  \end{cases}	
  \end{equation}
  $\partial_0Q$ denoting the parabolic boundary of $Q$. The coefficients satisfy in $Q$ all the classical assumptions (see A. Friedman  \cite{AF_book} Theorem 9 and Corollary 2 in Chapter 3, Sec. 4), so a unique (bounded) solution $v\in \mcC^{1,2}([0,S)\times\mcR)\cap \mcC([0,S]\times\bar{\mcR})$ actually exists (and have Hölder continuous derivatives $v_t$, $D_{(x,y)} v$ and $D^2_{(x,y)}v$ in any compact set contained inside $Q$). In Proposition \ref{comparison_principle} we proved a comparison principle on the unbounded spatial domain $\R\times\R_+$ it is easy (even easier) to prove with the same techniques a comparison principle for the problem above \eqref{parabolic_problem_bounded_domain}: this time the ``strictificate'' the super-solution $v$ can just add the term $\phi_\varepsilon= \varepsilon e^{Mt}$. So $u$ that is a viscosity solution and $v$ that is a classic one (hence a viscosity one, too) must be the same over $\overline{Q}$. Then we have the regularity claimed for $u$ over a generic point $(t,x,y)\in[0,T)\times(\R\times\R^*_+)$.
\end{proof}

\subsection{Discontinuous initial data}\label{discont_initial_data_subsec}
Here, we present a general result that allows us to reduce the regularity of the initial data and consider functions that are not continuous everywhere. 

Let $f$ be a locally bounded function defined over a locally compact domain $\mcO$. We define the upper-semicontinuous envelope $f^*$ and the lower-semicontinuous envelope $f_*$ on $\bar{\mcO}$ by
  $$
  f^*(x) = \limsup_{z\rightarrow x} f(z), \quad f_*(x) = \liminf_{z\rightarrow x} f(z). $$
We call $D_f$ the set of the discontinuity points of $f$, i.e.
\begin{equation}\label{discontinuity_set}
  D_f = \{x\in\bmcO\mid f^*(x)\neq f_*(x) \}.
\end{equation} 
Now, we state and prove a result that says the lower semicontinuous and upper semicontinuous functions belong to the Baire class 1  functions (i.e. they are the pointwise limit of continuous functions) and can be approximated in a monotone way.


\begin{prop}[Baire]\label{baire_approx}
  Let $X\subset\R^d$ closed. Let $f\in\LSC(X)$ then there exist a non-decreasing sequence $(f^-_n)_{n\in\N_*}\subset C(X)$ such that $f^-_n\rightarrow f$. If $f\in\USC(X)$ then there exist a non-increasing sequence $(f^+_n)_{n\in\N_*}\subset C(X)$ such that $f^+_n\rightarrow f$.
\end{prop}
\begin{proof}
  We prove only the first result; the second follows from the first, considering $-f$.
  We now sketch the proof.
  Let $f\le 0$, for $f$ not bounded below, we can find a continuous function $h\le f$, and we apply what follows to $\hf=f-h$, find approximations $\hf_n$ and define $f_n=\hf_n+h$.
  We consider the continuous function $g:x\rightarrow x/(1+x)$, prove the result for $\tf= g(f)$ taking values in $[0,1]$ (closed set) as in the proof in Proposition 11 in Section 2 of Chapter 9 of \cite{Bourbaki66}, consider the approximations $\tf_n$ given by proving the result for $\tf$ and then name $f_n=\tf_n/(1-\tf_n)$.
\end{proof}

We want to show that $(t,x,y)\mapsto \E[e^{\varrho(T-t)}f(X_T^{t,x,y},Y^{t,y}_T)]$ is a continuous mapping for $t<T$, even when the final data is discontinuous.
To this purpose, we use the fact that the distribution of the couple $(X^{t,x,y}_T,Y^{t,y}_T)$ for $t<T$ is absolutely continuous (with respect to Lebesgue measure). In fact, in \cite{PP2015}, it has been shown that $(\exp(X^{t,x,y}_T),Y^{t,y}_T)$ has a $\mcC^\infty(\R^*_+\times\R^*_+)$ density so using a change of variable argument, it easily follows that $(X^{t,x,y}_T,Y^{t,y}_T)$ has a $\mcC^\infty$ density over $\R\times\R^*_+$. 

\begin{prop}\label{continuity_prop}
  Let $f:\R\times\R_+\mapsto\R$ be a function such that the closure $\bDf$ of the set of its discontinuity points has zero Lebesgue measure. We assume that for all compact set $\mcK_T\in[0,T]\times\R\times\R_+$ there exists $p>1$ such that 
  $$
  \sup_{(t,x,y)\in\mcK_T}\|f(X_T^{t,x,y},Y^{t,y}_T)\|_{L^p(\Omega)}<\infty.
  $$ 
  Then $v(t,x,y)=\E[e^{\varrho(T-t)}f(X_T^{t,x,y},Y^{t,y}_T)]\in \mcC\big(([0,T]\times\R\times\R_+)\backslash (\{T\}\times \bDf) \big)$.
\end{prop}

\begin{proof}
  Since $(t,x,y)\mapsto e^{\varrho(T-t)}$ is continuous, we can consider the case $\varrho=0$. 
  We consider $(t,x,y)\in[0,T)\times\R\times\R_+$ and a sequence $(t_k,x_k,y_k)$ that converges towards it, we want to show that 
  \begin{equation*}
    |\E[f(X_T^{t_k,x_k,y_k},Y^{t_k,y_k}_T)] - \E[f(X_T^{t,x,y},Y^{t,y}_T)]| \xrightarrow[k\rightarrow\infty]{} 0.
  \end{equation*}
  Let $\varepsilon>0$. First, thanks to Hölder inequality
  \begin{multline*}
    |\E[f(X_T^{t_k,x_k,y_k},Y^{t_k,y_k}_T)\mathds{1}_{|(X_T^{t_k,x_k,y_k},Y^{t_k,y_k}_T)|>R}] |\le \\
     || f(X_T^{t_k,x_k,y_k},Y^{t_k,y_k}_T)||_{L^p(\Omega)} \P(|(X_T^{t_k,x_k,y_k},Y^{t_k,y_k}_T)|>R)^{\frac{p}{p-1}}.
  \end{multline*}
  One can choose $R$ such that the same inequality holds replacing $(X_T^{t_k,x_k,y_k},Y^{t_k,y_k}_T)$ with $(X_T^{t,x,y},Y^{t,y}_T)$. So, for $R$ big enough, one has
  \begin{multline*}
    |\E[f(X_T^{t_k,x_k,y_k},Y^{t_k,y_k}_T)] - \E[f(X_T^{t,x,y},Y^{t,y}_T)]| \le \\
    |\E[f(X_T^{t_k,x_k,y_k},Y^{t_k,y_k}_T)\mathds{1}_{|(X_T^{t_k,x_k,y_k},Y^{t_k,y_k}_T)|\le R}]- \E[f(X_T^{t,x,y},Y^{t,y}_T)\mathds{1}_{|(X_T^{t,x,y},Y^{t,y}_T)|\le R}]| +2\varepsilon.
  \end{multline*}
  Let $\delta>0$, we define the compact set $A^\delta_R=\{z\in\R\times\R_+ \mid d(z,\bDf)\ge \delta, |z|\le R \}$ and $C^\delta_R=\{z\in\R\times\R_+ \mid d(z,\bDf)< \delta, |z|\le R \}$. $C^\delta_R \searrow \bDf\cap \overline{B_R(0)}$ that is a null set, so thanks to the absolute continuity of $(X_T^{t,x,y},Y^{t,y}_T)$ one has 
  \begin{equation}\label{conv_prob_0}
    \P((X_T^{t,x,y},Y^{t,y}_T)\in C^\delta_R)\rightarrow 0 \text{ when } \delta\rightarrow0.
  \end{equation}
  Furthermore, the boundary $\partial C^\delta_R$ is a null set. In fact, it is contained in $\bDf \cup \partial B_R(0)\cup \{ z\in \R\times\R_+ \mid d(z,\bDf)= \delta \}$ that are three null sets, the first by hypothesis and the second two being sets whose points are exactly distant a strictly positive number from closed sets ($\{0\}$ and $\bDf$) (look here \cite{Erdos1945} for a simple proof). So, by convergence in distribution of $(X_T^{t_k,x_k,y_k},Y^{t_k,y_k}_T)$ towards $(X_T^{t,x,y},Y^{t,y}_T)$
  \begin{equation}\label{conv_distrib}
    \P((X_T^{t_k,x_k,y_k},Y^{t_k,y_k}_T)\in C^\delta_R)\xrightarrow[k\rightarrow \infty]{} \P((X_T^{t,x,y},Y^{t,y}_T)\in C^\delta_R).
  \end{equation}
  Thanks to the following inequality
  \begin{multline*}
    |\E[f(X_T^{t_k,x_k,y_k},Y^{t_k,y_k}_T)\mathds{1}_{(X_T^{t_k,x_k,y_k},Y^{t_k,y_k}_T)\in C^\delta_R}] | \le \\
     || f(X_T^{t_k,x_k,y_k},Y^{t_k,y_k}_T)||_{L^P(\Omega)} \P((X_T^{t_k,x_k,y_k},Y^{t_k,y_k}_T)\in C^\delta_R)^{\frac{p}{p-1}},
  \end{multline*}
  that still valid replacing $(t_k,x_k,y_k)$ with $(t,x,y)$, using the uniform boundedness in $L^P$ hypothesis, \eqref{conv_prob_0} and \eqref{conv_distrib}, if $\delta$ is small enough and $k\ge k_0$ one has
  \begin{equation*}
    |\E[f(X_T^{t_k,x_k,y_k},Y^{t_k,y_k}_T)\mathds{1}_{(X_T^{t_k,x_k,y_k},Y^{t_k,y_k}_T)\in C^\delta_R}] - \E[f(X_T^{t,x,y},Y^{t,y}_T)\mathds{1}_{(X_T^{t,x,y},Y^{t,y}_T)\in C^\delta_R}]|\le 2\varepsilon.
  \end{equation*}
  The set $\partial A^\delta_R$ is a null set, because it is contained in $\partial B_R(0)\cup \{ z\in \R\times\R_+ \mid d(z,\bDf)= \delta \}$, that are two null set, as explained above.
  So, the function $f\mathds{1}_{A^\delta_R}$ are continuous over the compact $A^\delta_R$, therefore, they are bounded and are discontinuous only over the null set $\partial A^\delta_R$. Thanks to the absolute continuity of $(X_T^{t,x,y},Y^{t,y}_T)$ and convergence in distribution of $(X_T^{t_k,x_k,y_k},Y^{t_k,y_k}_T)$ towards it, if $k\ge k_1$
  $$
  |\E[f(X_T^{t_k,x_k,y_k},Y^{t_k,y_k}_T)\mathds{1}_{(X_T^{t_k,x_k,y_k},Y^{t_k,y_k}_T)\in A^\delta_R}]- \E[f(X_T^{t,x,y},Y^{t,y}_T)\mathds{1}_{(X_T^{t,x,y},Y^{t,y}_T)\in A^\delta_R}]| \le \varepsilon
  $$
  Finally if $k\ge \max(k_0,k_1)$ one has 
  $$
  |\E[f(X_T^{t_k,x_k,y_k},Y^{t_k,y_k}_T)] - \E[f(X_T^{t,x,y},Y^{t,y}_T)]| \le 5\varepsilon.
  $$
  If we consider $(t,x,y)\in\{T\}\times A$, where $A=(\R\times\R_+)\backslash\bDf$ then for any sequence $(t_k, x_k, y_k)_{k\in\N}$ that convergences to $(t,x,y)$ and all the previous estimation still works because the limit law this time is a Dirac mass over $(t,x,y)$ that is not a discontinuity point of $f$.
\end{proof}

Thanks to the previous proposition, we can consider some type of discontinuities in the final data $f$ (or initial data if we consider the forward problem). We state and prove the following result.
\begin{theorem}\textbf{(Verification Theorem)}\label{verification_theorem}
  Let $f:\R\times\R_+\rightarrow\R$ be a polynomial growth function such that $\bDf$ has zero Lebesgue measure, and $h\in \mcC_{\pol,T}(\R\times\R_+)$.
  Then $u$ in \eqref{u_sol_chap3} is the unique viscosity solution to the problem \eqref{reference_PDE} that is $\mcC\big(([0,T]\times\R\times\R_+)\setminus (\{T\}\times \bDf) \big)$ and that has polynomial growth in $(x,y)$ uniformly in $t$. 
\end{theorem}

\begin{proof}
  Let $u$ be as in \eqref{u_sol_chap3} that is a function in the class considered, and let $v$ a solution in this same class. We use an approach that directly proves that $u$ is a solution and is the only one in the class considered.
  We show that exist a continuous sequence $(u^-_n)_{n\in\N}$ of sub-solution and a continuous sequence $(u^+_n)_{n\in\N}$ of super-solution such that
  \begin{align}
    &u^-_n(T,\cdot,\cdot)\le v_*(T,\cdot,\cdot)\le v^*(T,\cdot,\cdot)\le u^+_n(T,\cdot,\cdot), \label{squeeze_v} \\
    \text{for any compact } &\text{set } \mcK_T\subset[0,T)\times\R\times\R_+, \text{ one has }\lim_{n\rightarrow\infty} |u^\pm_n-u|_0^{\mcK_T}  \label{local_uniform_conv_VT},
  \end{align}
  where $v_*,v^*$ are respectively the lower and the upper semicontinuous envelope of $v$. 
  The local uniform convergence \eqref{local_uniform_conv_VT} tells us, thanks to Lemma \ref{sequence_lemma_general}, the limit $u$ is both sub and a super-solution, so it is a solution. 
  Then, if one has the inequalities \eqref{squeeze_v}, one can apply the comparison principle (reverting in time the solutions) comparing $u^-_n$ to $v_*$, and $u^+_n$ to $v^*$ getting the relations
  \begin{align*}
    u(t,x,y)&=\lim_{n\rightarrow\infty}u^+_n(t,x,y)\ge v^*(t,x,y)=v(t,x,y),\quad \text{ for all } t<T, \, (x,y)\in\R\times\R_+,\\
    u(t,x,y)&=\lim_{n\rightarrow\infty}u^-_n(t,x,y)\le v_*(t,x,y)=v(t,x,y),\quad \text{ for all } t<T, \, (x,y)\in\R\times\R_+,
  \end{align*}
  where we used the fact that $v_*(t,x,y)=v(t,x,y)=v^*(t,x,y)$ because $v$ is continuous for every point such that $t<T$. Then, $u=v$ everywhere because they have the same final data.
  We prove, so, the existence of continuous sequences of solutions that satisfy \eqref{squeeze_v}. We consider a modified version of final data $f^\pm$ defined as follows
  \begin{equation}
    \begin{cases}
      f^\pm(x,y)= \pm \chi(x,y) &\text{for }(x,y)\in N \\
      f^\pm(x,y)= f(x,y) &\text{otherwise},
      \end{cases}
  \end{equation}
  where $\chi(x,y)=C(1+|x|^L+y^L)$ $(C>0,L\in\N^*)$ is such that $|u(t,x,y)|,|v(t,x,y)|\le \chi(x,y)$, and we define the functions $u^\pm$
  $$
  u^\pm(t,x,y)=\E\left[e^{\varrho (T-t)}f^\pm(X_T^{t,x,y},Y^{t,y}_T)-\int_t^T e^{\varrho (s-t)} h(s,X^{t,x,y}_s,Y^{t,y}_s)ds\right].
  $$
  Being $f^-$ and $f^+$ respectively lower semicontinuous and upper semicontinuous, thanks to Proposition $\ref{baire_approx}$, there exists 
  $(f^-_n)_{n\in\N_*}\subset C(\R\times\R_+)$ non-decreasing sequence converging to $f^-$ and such that $f^-_n\ge-\chi$ (otherwise consider $\hat{f}^-_n=f^-_n\vee -\chi$) and $(f^+_n)_{n\in\N_*}\subset C(\R\times\R_+)$ non-increasing sequence converging to $f^+$ and such that $f^+_n\ge\chi$ (otherwise consider $\hat{f}^+_n=f^+_n\wedge \chi)$. We define in the same way as $u^\pm$
  $$
  u^\pm_n(t,x,y)=\E\left[e^{\varrho (T-t)}f^\pm_n(X_T^{t,x,y},Y^{t,y}_T)-\int_t^T e^{\varrho (s-t)} h(s,X^{t,x,y}_s,Y^{t,y}_s)ds\right].
  $$
  By Proposition \ref{existence_viscsol_fcont} $u^+_n$ are continuous solution with final data $f^+_n\ge f\ge v_*(T,\cdot,\cdot)$ and $u^-_n$ are continuous solution with final data $f^-_n\le f\le v^*(T,\cdot,\cdot)$ (proving \eqref{squeeze_v}). Furthermore, thanks to Lebesgue Theorem $u^\pm_n\rightarrow u^\pm$ when $n\rightarrow\infty$ and $u^\pm=u$ for $t<T$ because $f$ differs from $f^\pm$ only on a negligible set and $(X^{t,x,y}_T,Y^{t,y}_T)$ has density, and the convergence is monotone in $n$ for any point. Then, considering the convergence on any compact set $\mcK_T\subset[0,T)\times\R\times\R_+$ we have a monotone sequence  ($u^+_n$ or $u^-_n$) of continuous functions that converges everywhere over $\mcK_T$ to a continuous function $u$, so this convergence must be uniform for Dini's Theorem (proving \eqref{local_uniform_conv_VT}).
\end{proof}

We conclude with one remark and an example.

\begin{remark}\label{remark_regularity_f_discont}
  One should remark that if we add in Theorem \ref{verification_theorem} that $h$ is locally Hölder in the compact sets of $[0,T)\times\R\times\R^*_+$, then $u$ belongs also to $\mcC^{1,2}\big([0,T)\times(\R\times\R^*_+)\big)$.
  The proof of the regularity is the same as in Proposition \ref{uniqueness_and_regularity_f_cont}, so we do not repeat here.
\end{remark}

\begin{ex}\label{digital_option}
  Theorem \ref{verification_theorem} covers a wide range of possibilities. One of them is the case of digital options in the Heston Model. We fix the parameters $c = r-\delta,$ $d=-1/2$, $\lambda=1$ and $\varrho =r$
  \begin{equation*}
    \begin{cases}
      \partial_tu(t,x,y) +\mcL u(t,x,y) +r u(t,x,y) = 0,\quad   t\in[0,T), &x\in\R, y\in\R_+, \\
      u(T,x,y) = \ind{[c,d)}(\exp(x)), \quad  &x\in\R, y\in \R_+,
    \end{cases}
  \end{equation*}
  where $0\le c < d \le\infty$.
\end{ex}

%

\section{Application to finance: a hybrid approximation scheme for the viscosity solution}\label{approximation_section}
Consider the standard Heston model given by the following SDE
\begin{align}\label{Heston_model_sde}
  S_T^{t,s,y} &= s + \int_t^T (r-\delta)S_u du+\int_t^T\rho S_u\sqrt{Y^{t,y}_u} dW_u +\int_t^T\brho S_u\sqrt{Y^{t,y}_u} dB_u, \nonumber\\
  Y^{t,y}_T &= y + \int_t^T(a-bY^{t,y}_u)du + \int_t^T\sigma\sqrt{Y^{t,y}_u}dW_u.
\end{align}
In order to build our approximation, we  apply the transformation $(s,y)\mapsto (\log(s)-\frac{\rho}{\sigma}y,y)$ obtaining the following SDE
\begin{align}\label{logHestonDiffusion_nocorr}
  X_T^{t,x,y} &= x + \int_t^T\Big(r-\delta-\frac{\rho}{\sigma}a+\big(\frac{\rho}{\sigma}b-\frac{1}{2}\big)Y^{t,y}_s\Big) ds +\int_t^T\bar{\rho}\sqrt{Y^{t,y}_s} dB_s, \nonumber\\
  Y^{t,y}_T &= y + \int_t^T(a-bY^{t,y}_s)ds + \int_t^T\sigma\sqrt{Y^{t,y}_s}dW_s,
\end{align}
that given corresponds to a precise choice of the parameters in \eqref{referenceDiffusion}: $\rho=0$, $c=r-\delta-\frac \rho{\sigma}a$, $d=\frac \rho{\sigma}b-\frac 12$ and $\lambda = \brho$. 
The advantage of studying \eqref{logHestonDiffusion_nocorr} instead of the SDE obtained by $(s,y)\mapsto (\log(s),y)$ is that we can exploit that the noise driving the law of $X|Y$ is independent of the one driving $Y$.  
Hereafter, we fix $T>0$, $f\in\mcC_{\pol}(\R\times\R_+)$ and define
\begin{equation}\label{u_solution}
  u(t,x,y)=\E[f(X^{t,x,y}_T,Y^{t,y}_T)],\qquad(t,x,y)\in[0,T]\times\R\times\R_+.
\end{equation}
We know that
\begin{enumerate}
  \item $\P((X_s^{t,x,y},Y_s^{t,y})\in\R\times\R_+,\forall s\in[t,T])=1;$
  \item the function $u$ in \eqref{u_solution} solves the PDE
      \begin{equation}\label{logHeston_PDE_nocorr}
        \begin{cases}
          (\partial_t+\mcL)u(t,x,y) = 0,\quad   t\in[0,T), &x\in\R, y\in\R_+, \\
          u(T,x,y) = f(x,y), \quad  &x\in\R, y\in \R_+,
        \end{cases}
      \end{equation}
      where 
      \begin{equation}\label{Heston_infinit_gen_nocorr}
        \mcL = \frac{y}{2}(\brho^2 \partial^2_{x} +  \sigma^2 \partial^2_{y}) +\mu_X(y) \partial_x + \mu_Y(y) \partial_y,
      \end{equation}
      and  $\mu_X(y)= r-\delta-\rho a/\sigma +(\rho b/\sigma-1/2) y$,  $\mu_Y(y)=a-by$.
\end{enumerate}
In what follows, we prove the convergence for a large space of functions using the recent hybrid approach introduced in \cite{BCT} that we recall in what follows.

\subsection{The hybrid procedure}
Let $u$ be given in \eqref{u_solution}. We recall, briefly, the main ideas and describe the approximation of $u$. We use the Markov property to represent the solution $u(nh,x,y)$ at times $nh$, $h=T/N$, $n=0,\ldots,N$ for $(x,y)\in\R\times\R_+$
\begin{equation}\label{dynamic_programming_principle}
  \begin{cases}
    u(T,x,y)=f(x,y) \quad \text{ and } n=N-1,\ldots,0: \\
    u(nh,x,y)= \E[u((n+1)h,X^{nh,x,y}_{(n+1)h},Y^{nh,y}_{(n+1)h})].
  \end{cases}
\end{equation}

The goal is to build good approximations of the expectations in \eqref{dynamic_programming_principle}.
First, let $(\hY^h_n)_{n=0,\ldots,N}$ be a Markov chain which approximates $Y$, such that $(\hY^h_n)_{n=0,\ldots,N}$ is independent of the noise driving $X$. 
Then, at each step $n=0,1,\ldots,N-1$, for every $y\in\mcY^h_n\subset\R_+$ (the state space of $\hY^h_n$), one writes
$$
  \E[u((n+1)h,X^{nh,x,y}_{(n+1)h},Y^{nh,y}_{(n+1)h})] \approx \E[u((n+1)h,X^{nh,x,y}_{(n+1)h},\hY^h_{n+1})\mid \hY^h_n=y].
$$
As a second step, one approximates the component $X$ on $[nh,(n+1)h]$ by freezing the coefficients in \eqref{logHestonDiffusion_nocorr} at the observed position $\hY^h_n=y$, that is, for $t\in[nh,(n+1)h]$,
$$
X^{nh,x,y}_t \overset{\text{law}}{\approx} \hX^{nh,x,y}_t = x + \Big(r-\delta-\frac{\rho}{\sigma}a+\big(\frac{\rho}{\sigma}b-\frac{1}{2}\big)y\Big)(t-nh)+\brho\sqrt{y}(Z_t-Z_{nh}).
$$
Therefore, by the fact that the Markov chain and the noise driving $X$ are independent, one can write
\begin{align*}
  \E[u((n+1)h,X^{nh,x,y}_{(n+1)h},Y^{nh,y}_{(n+1)h})] &\approx \E[u((n+1)h,\hX^{nh,x,y}_{(n+1)h},\hY^{h}_{n+1})\mid \hY^h_n=y] \\
  &=\E[\phi(\hY^h_{n+1};x,y)\mid \hY^h_n=y]
\end{align*}
where
\begin{equation}
  \phi(\zeta;x,y)=\E[u((n+1)h,\hX^{nh,x,y}_{(n+1)h},\zeta)].
\end{equation}
From the Feynman-Kac formula, one gets $\phi(\zeta;x,y)=v(nh,x;y,\zeta)$, where $(t,x)\mapsto v(t,x;y,\zeta)$ is the solution at time $nh$ of the parabolic PDE Cauchy problem
\begin{equation}\label{PDE_frozencoeff}
  \begin{cases}
    \partial_t v +\mcL^{(y)}v=0, &\text{in } [nh,(n+1)h)\times\R,\\
    v((n+1)h,x;y,\zeta) = u((n+1)h,x,\zeta), &x\in\R,
  \end{cases}
\end{equation}
where $\mcL^{(y)}$ acts on function $g=g(x)$ as follows
\begin{equation}\label{mcLy_frozencoeff}
  \mcL^{(y)} g(x)= \Big(r-\delta-\frac{\rho}{\sigma}a+\big(\frac{\rho}{\sigma}b-\frac{1}{2}\big)y\Big)\partial_x g(x) + \frac{1}{2}\brho^2y\partial^2_xg(x).
\end{equation}
We remark that in \eqref{PDE_frozencoeff}-\eqref{mcLy_frozencoeff}, $y\in\R_+$ is just a parameter, so $\mcL^{(y)}$ has constant coefficients.
Consider now a numerical solution of the PDE \eqref{PDE_frozencoeff}. Let $\Dx$ denote a fixed initial spatial step, and set $\mcX$ as a grid on $\R$ given by $\mcX=\{x\in\R \mid x=X_0+i\Dx, i\in\Z \}$. For $y\in\R$, let $\PhDx(y)$ be a linear operator (acting on suitable functions on $\mcX$) which gives the approximating solution to the PDE \eqref{PDE_frozencoeff} at time $nh$. Then, as $x\in\mcX$, we get the numerical approximation
$$
\E[u((n+1)h,X^{nh,x,y}_{(n+1)h},Y^{nh,y}_{(n+1)h})] \approx \E[\PhDx(y)u((n+1)h,\cdot,\hY^{h}_{n+1})(x)\mid \hY^h_n=y].
$$
Therefore, by inserting in \eqref{dynamic_programming_principle}, one sees that the hybrid numerical procedure works as follows: the function $x\mapsto u(0,x,Y_0)$, $x\in\mcX$, is approximated by $u^h_0(x,Y_0)$ backward-defined as 
\begin{equation}\label{hybrid_scheme}
  \begin{cases}
    u^h_N(x,y) = f(x,y), \qquad (x,y)\in \mcX\times\mcY^h_N,\quad \text{and as } n=N-1,\ldots,0: \\
    u^h_n(x,y) = \E[\PhDx(y) u^h_{n+1}(\cdot,\hat{Y}^h_{n+1})(x) \mid \hat{Y}^h_{n}=y], \quad (x,y)\in\mcX\times\mcY^h_n.
    \end{cases}
\end{equation}

\subsection{Convergence in $\ell^\infty$}
We recall the finite difference scheme and the Markov chain $(\hY^h_n)_{n=0,\ldots,N}$, that under suitable hypothesis on $f$ assures the convergence of the Hybrid procedure to the solution $u$ \eqref{u_solution}.
Specifically, if $\mu_X(y)=\frac{h}{\Dx}r-\delta-\frac{\rho}{\sigma}a+\big(\frac{\rho}{\sigma}b-\frac{1}{2}\big)y\geq 0$, we approximate  $(\partial_t+\mcL^{(y)})v$ by using the scheme 
$$
\frac{v^{n+1}_i-v^n_i}{h}+ \mu_X(y)\frac{v^{n}_{i+1}-v^n_i}{\Dx} + \frac{1}{2}\brho^2y\frac{v^{n}_{i+1}-2v^n_i+v^n_{i-1}}{\Dx^2} ,
$$
while, if $\mu_X(y)\leq 0$, we use the approximation
$$
\frac{v^{n+1}_i-v^n_i}{h}+ \mu_X(y)\frac{v^{n}_{i}-v^n_{i-1}}{\Dx} + \frac{1}{2}\brho^2y\frac{v^{n}_{i+1}-2v^n_i+v^n_{i-1}}{\Dx^2}.
$$
The resulting scheme is 
\begin{equation}\label{equaz2}
A^h_{\Dx}(y)v^n=v^{n+1},
\end{equation}
where $A^h_{\Dx}(y)$ is the linear operator given by  
\begin{equation}\label{A2}
(A^h_{\Dx})_{ij}(y)=\begin{cases}
-\beta^h_{\Dx}(y)-|\alpha^h_{\Dx}(y)|\ind{\alpha^h_{\Dx}(y)<0},\qquad &\mbox{ if }i=j+1,\\
1+2\beta^h_{\Dx}(y)+|\alpha^h_{\Dx}(y)|,\qquad &\mbox{ if }i=j,\\
-\beta^h_{\Dx}(y)-|\alpha^h_{\Dx}(y)|\ind{\alpha^h_{\Dx}(y)>0},\qquad &\mbox{ if }i=j-1,\\
0, &\mbox{ if }|i-j|>1,
\end{cases}
\end{equation}  
with
$$
\alpha^h_{\Dx}(y)=\frac{h}{\Dx}r-\delta-\frac{\rho}{\sigma}a+\big(\frac{\rho}{\sigma}b-\frac{1}{2}\big)y, \qquad \beta^h_{\Dx}(y)=\frac{h}{2\Dx^2}y.
$$
We finally define $\PhDx(y)=\big(A^h_{\Dx}(y)\big)^{-1}$. 

Along with the finite different scheme, we need a Markov chain $(\hY^h_n)_{n=0,1,\ldots,N}$ approximating the CIR process $Y$ over the time grid $(nT/N)_{n=0,1,\ldots,N}$. The state space is the following, for $n=0,1,\ldots,N$ one has the lattice
\begin{equation}\label{vnk}
\mathcal{Y}_n^h=\{y^n_k\}_{k=0,1,\ldots,n}\quad\mbox{with}\quad
y^n_k=\Big(\sqrt {y}+\frac{\sigma} 2(2k-n)\sqrt{h}\Big)^2\ind{\{\sqrt {y}+\frac{\sigma} 2(2k-n)\sqrt{h}>0\}}.
\end{equation}
Note that $\mathcal{Y}_0^h=\{y\}$. For each fixed node $(n,k)\in\{0,1,\ldots,N-1\}\times\{0,1,\ldots,n\}$, the ``up'' jump $k_u(n,k)$ and the ``down'' jump $k_d(n,k)$ from $y^n_k\in\mathcal{Y}_n^h$ are defined as 
\begin{align}
\label{ku2}
&k_u(n,k) =
\min \{k^*\,:\, k+1\leq k^*\leq n+1\mbox{ and }y^n_k+\mu_Y(y^n_k)h \le y^{n+1}_{k^*}\},\\
\label{kd2}
&k_d(n,k)=
\max \{k^*\,:\, 0\leq k^*\leq k \mbox{ and }y^n_k+\mu_Y(y^n_k)h \ge y^{n+1}_{ k^*}\},
\end{align}
where $\mu_Y(y)=a-by$ and
with the understanding $k_u(n,k)=n+1$, resp. $k_d(n,k)=0$,  if the set in \eqref{ku2}, resp. \eqref{kd2}, is empty. 
Starting from the node $(n,k)$ the probability that the process jumps to $k_u(n,k)$ and $k_d(n,k)$ at time-step $n+1$ are set respectively as
\begin{equation*}
p_u(n,k)
=0\vee \frac{\mu_Y(y^n_k)h+ y^n_k-y^{n+1}_{k_d(n,k)} }{y^{n+1}_{k_u(n,k)}-y^{n+1}_{k_d(n,k)}}\wedge 1
\quad\mbox{and}\quad p_d(n,k)=1-p_u(n,k).
\end{equation*}
We call $(\hY^h_n)_{n=0,1,\ldots,N}$ the Markov chain governed by the above jump probabilities.

Let $\ell^\infty(\mcX)=\{g:\mcX\rightarrow\R\mid \sup_{x\in\mcX} g(x)<\infty\}$ with the norm $|g|_{\ell^\infty}=\sup_{x\in\mcX} g(x)$.
With the above Markov chain, Briani et al. in \cite{BCT} proved that $\PhDx(\cdot)$ satisfies the following Assumption $\mcK$ with $c=1$ and $\mcE=h+\Dx$.

\begin{definition}[\textbf{Assumption} $\mcK(\infty,c,\mcE)$]
  Let $c=c(y)\ge 0$, $y\in\R_+$, and $\mcE=\mcE(h,\Dx)\ge 0$ such that $\lim_{(h,\Dx)\rightarrow0}\mcE(h,\Dx)=0$. We say that the linear operator $\PhDx(y):\ell^\infty(\mcX)\rightarrow\ell^\infty(\mcX)$, $y\in\mcD$, satisfies this assumption if 
\begin{equation}
  \|\PhDx(y)\|_\infty:= \sup_{|f|_{\ell^\infty}=1}|\PhDx(y)f|_{\ell^\infty}\le 1+c(y)h, 
\end{equation}
and, with $u$ being defined in \eqref{u_solution}, for every $n=0,\ldots,N-1$, one has
\begin{equation}
  \E[\PhDx(\hY^h_n)u((n+1)h,\cdot,\hY^{h}_{n+1})(x)\mid Y^h_n] = u(nh,x,\hY^h_n)+\mcR^h_n(x,\hY^h_n),
\end{equation}
where the remainder $\mcR^h_n(x,\hY^h_n)$ satisfies the following property: there exists $\bar{h},C>0$ such that for every $h<\bar{h},\Dx<1$, and $n\le N=\lfloor T/h \rfloor$ one has
\begin{equation}
  \left\|e^{\sum_{l=1}^n c(\hY^h_l)h}|\mcR^h_n(\cdot,\hY^h_n)|_{\ell^\infty} \right\|_{L^1(\Omega)} \le C h\mcE(h,\Dx).
\end{equation}
\end{definition}

In \cite{BCT}, Briani et al. proved the following Theorem.
\begin{theorem}\label{BCT_regular_approx}
  Let $u$ defined in \eqref{u_solution}, $(u^h_n)_{n=0,\ldots,N}$ be given by \eqref{hybrid_scheme} with the choice
  $$
  \PhDx(y)=\big(A^h_{\Dx}(y)\big)^{-1},
  $$
  where $A^h_\Dx(y)$ is given in \eqref{A2}, and $(\hY^h_n)_{n=0,1,\ldots,N}$ defined as above.
  If $\partial_x ^{2j}f\in \mcC^{\infty,q-j}_{\pol}(\R,\R_+)$, for every $j=0,1,\ldots,4$, then, there exist $\bar{h},C>0$ such that for every $h<\bar{h}$ and $\Dx<1$ one has 
  \begin{equation}
    |u(0,\cdot,y) - u^h_0(\cdot,y)|_{\ell^\infty} \le C(h+\Dx).
  \end{equation}
\end{theorem}

We are about to show that, under less regular $f$, we still have the convergence of the hybrid procedure. We present a lemma that will help us prove this result.
\begin{lemma}\label{stability_conv_Cpol}
  Let $f\in \mcC^{\infty,0}_{\pol}(\R,\R_+)$, $(\varphi_l)_{l\in\N^*}$ a sequence of mollifiers over $\R^2$ and
  \begin{equation}
    \tf(x,y) = f(x,0\vee y) \quad \text{and} \quad f_l=\tf\ast \varphi_l.
  \end{equation}
  Then, for all $q\in\N$, $f_l\in \mcC^{\infty,q}_{\pol}(\R,\R_+)$. In particular  $\exists C_0,C^*_0 >0$ such that for all $l\in\N$, $x\in\R$ and $y\in\R_+$
  \begin{equation}\label{const_stability_conv}
    |f_l(x,y)|\le C_0(1+|x|^L+y^L), \quad \sup_{x\in\R}|f_l(x,y)|\le C^*_0(1+y^L).
  \end{equation}
\end{lemma}
\begin{proof}
  Let $q\in\N$. Let $f\in \mcC^{\infty,0}_{\pol}(\R,\R_+)$, then $\tf\in \mcC^{\infty,0}_{\pol}(\R,\R)$ and in particular $\tilde{f}\in\mathbb{L}^\infty_{\text{loc}}(\R^2)$, hence $f_l\in C^\infty(\R^2)$ and, for all multi-index $k$, $D^k f_l = \tilde{f}*D^k\varphi_l$. It remains for us to prove the polynomial growth of $D^kf_l$ and $\sup_{x\in\R}|D^k f_l(x,y)|$ for all multi-index $k$ such that $|k|\le q$. Let $x\in\R,~y\in\R_+$, using that $\tf\in \mcC^{\infty,0}_{\pol}(\R,\R)$ and $(a+b)^L\le 2^{L-1}(a^L+b^L)$ for all $a,b\ge0$
  \begin{align*}
    |D^k f_l(x,y)|&= \left|\int_{\R^2}\tilde{f}(\zeta-x,\eta-y)D^k\varphi_l(\zeta,\eta)d\zeta d\eta \right| \le \int_{\R^2}|\tilde{f}(\zeta-x,\eta-y)| |D^k\varphi_l(\zeta,\eta)|d\zeta d\eta \\
    &\le \int_{\R^2}C(1+|\zeta-x|^L+|\eta-y|^L) |D^k\varphi_l(\zeta,\eta)|d\zeta d\eta \\
    &\le C +C_L(|x|^L+y^L) +C_L \int_{\R^2}(|\zeta|^L+|\eta|^L) |D^k\varphi_l(\zeta,\eta)|d\zeta d\eta \\
    &\le C_k(l)(1 +|x|^L+y^L) \\
    &\le \Big(\max_{k \,\text{s.t.}\, |k|\le q} C_k(l)\Big) (1 +|x|^L+y^L),
  \end{align*}
  where we used the fact that $D^k\varphi_l$ is $C^\infty_c(\R^2)$.
  Furthermore, if $k=0$ than $|D^k\varphi_l(\zeta,\eta)|=\varphi_l(\zeta,\eta)$ and the integral $\int_{\R^2}(|\zeta|^L+|\eta|^L) \varphi_l(\zeta,\eta) d\zeta d\eta \rightarrow 0$ when $l$ goes to $\infty$. So the constant $C_0(l)=C_0$.
  Now, let $y\in\R_+$,
  \begin{align*}
    \sup_{x\in\R} |D^k f_l(x,y)|  &= \sup_{x\in\R} \left|  \int_{\R^2}\tilde{f}(\zeta-x,\eta-y) D^k\varphi_l(\zeta,\eta) d\zeta d\eta\right| \\
    & \le\int_{\R^2}\sup_{z\in\R}|\tilde{f}(z,\eta-y)| |D^k\varphi_l(\zeta,\eta)|d\zeta d\eta, \\
    &\le (C+C_L y^L) +C_L \int_{\R^2}|\eta|^L |D^k\varphi_l(\zeta,\eta)|d\zeta d\eta \\
    &\le C^*_k(l)(1+y^L) \\
    &\le \Big(\max_{k \,\text{s.t.}\, |k|\le q} C^*_k(l)\Big)  (1+y^L),
  \end{align*}
  where once again, we used the fact that $D^k\varphi_l$ is $C^\infty_c(\R^2)$. As in the previous estimate, if $k=0$ than $|D^k\varphi_l(\zeta,\eta)|=\varphi_l(\zeta,\eta)$ and the integral $\int_{\R^2} |\eta|^L \varphi_l(\zeta,\eta) d\zeta d\eta \rightarrow 0$ when $l$ goes to $\infty$. So the constant $C^*_0(l)=C^*_0$.
\end{proof}

We state and prove the main contribution of this section.

\begin{theorem}\label{convergence_theorem}
  Let $f\in \mcC^{\infty,0}_{\pol}(\R,\R_+)$ and suppose that $f$ is uniformly continuous over the sets $\R\times[0,M]$ for all $M>0$. Let $u$ be the viscosity solution defined in \eqref{u_solution} and $u^h_0$ the discrete solution \eqref{hybrid_scheme} produced by the backward hybrid procedure starting from $f$.
  Then
  \begin{equation}
    u^h_0(\cdot,y) \xrightarrow[(h,\Dx)\rightarrow 0]{l^\infty} u(0,\cdot,y).
  \end{equation}
\end{theorem}
\begin{proof}
  Let $(\varphi_l)_{l\in\N^*}$ a sequence of mollifiers. For all $l\in\N^*$, we define  $u_l(t,x,y) = \E[f_l(X^{t,x,y}_T,Y^{t,y}_T)]$ where we replaced $f$ with a mollified final data $f_l=f\ast\varphi_l$, and $u^{h}_{n,l}$ as the discrete solution produced by the backward hybrid procedure starting from $f_l$. Let $x\in\mcX$ and $y\in\R_+$, then
  \begin{align*}
    |u(0,\cdot,y) - u^h_0(\cdot,y)|_{\ell^\infty} \le
    &\underbrace{|u(0,\cdot,y) - \hu_l(0,\cdot,y)|_{\ell^\infty}}_{I}
    +\underbrace{|u_l(0,\cdot,y)-u^{h}_{0,l}(\cdot,y)|_{\ell^\infty}}_{II} \\
    +&\underbrace{|u^{h}_{0,l}(\cdot,y) - u^h_0(\cdot,y)|_{\ell^\infty}}_{III}.
  \end{align*}
Thanks to Theorem \ref{BCT_regular_approx}, term $II$ can be upper bounded by $C_lT(h+\Dx)$ that goes to zero when $(h,\Dx)$ goes to zero (and does not depend on $x$).
Regarding the term $I$, defined $(X^{\cdot,y}_T,Y^{y}_T)=(X^{0,\cdot,y}_T,Y^{0,y}_T)$ one has
\begin{align*}
  |u(0,\cdot,y) - u_l(0,\cdot,y)|_{\ell^\infty} &\le |\E[f(X^{\cdot,y}_T,Y^{y}_T)-f_l(X^{\cdot,y}_T,Y^{y}_T)]|_{\ell^\infty} \\
  &\le |\E[\big(f(X^{\cdot,y}_T,Y^{y}_T)-f_l(X^{\cdot,y}_T,Y^{y}_T)\big) \mathds{1}_{Y^{y}_T\le M}]|_{\ell^\infty} \\
  &\quad+| \E[\big(f(X^{\cdot,y}_T,Y^{y}_T)-f_l(X^{\cdot,y}_T,Y^{y}_T)\big) \mathds{1}_{Y^{y}_T>M}]|_{\ell^\infty} \\
  &\le \sup_{x\in\R,y\in[0,M]}|f(x,y)-f_l(x,y) | \\
  &\quad+ \E\big[\big(|f(\cdot,Y^{y}_T)|_{\ell^\infty}+|f_l(\cdot,Y^{y}_T)|_{\ell^\infty}\big) \mathds{1}_{Y^{y}_T>M}\big].
\end{align*}
Using \eqref{const_stability_conv}, $\exists C>0$ such that $|f(\cdot,Y^{y}_T)|_{\ell^\infty}+|f_l(\cdot,Y^{y}_T)|_{\ell^\infty}\le C(1+(Y^{y}_T)^L)$, then using Holder inequality and then Markov inequality one has
\begin{equation}\label{estim_I}
  I \le \sup_{x\in\R,y\in[0,M]}|f(x,y)-f_l(x,y) |+ C\|(1+(Y^{y}_T)^L)\|_{L^p(\Omega)} \left(\frac{E[Y^{y}_T]}{M}\right)^{\frac{p-1}{p}}.
\end{equation}
Regarding the term $III$, by linearity of the conditional expectation and of the linear operator $\Pi^h_\Dx(y)$ for $n=0,\ldots,N-1$
$$
u^{h}_{n,l}(\cdot,y) - u^h_n(\cdot,y) = \E\big[\Pi^h_\Dx(\hat{Y}^h_{n}) \big( u^{h}_{n,l}(\cdot,\hat{Y}^h_{n+1}) - u^h_{n+1}(\cdot,\hat{Y}^h_{n+1}) \big) \mid \hat{Y}^h_{n}=y \big].
$$
Then we can rewrite the difference  $u^{h}_{0,l}(\cdot,y) - u^h_0(\cdot,y)$ can be seen as the discrete solution constructed starting from $f_l-f$. In fact, using the tower properties, we get by induction
\begin{align*}
  u^{h}_{0,l}(\cdot,y) - u^h_0(\cdot,y) = &\E\bigg[\prod_{j=0}^{N-1}\Pi^h_\Dx(\hat{Y}^h_{j}) \big( f_l(\cdot,\hat{Y}^h_{N}) - f(\cdot,\hat{Y}^h_{N}) \big) \mid \hat{Y}^h_{0}=y \bigg] \\
  = &\E\bigg[\prod_{j=0}^{N-1}\Pi^h_\Dx(\hat{Y}^h_{j}) \big( f_l(\cdot,\hat{Y}^h_{N}) - f(\cdot,\hat{Y}^h_{N}) \big)\bigg].
\end{align*}
Furthermore, \cite[Lemma 4.7]{BCT} guarantees  $\|\Pi^h_\Dx(y)\|_{\infty} \le 1$ and so
$$
\Big\|\prod_{j=0}^{N-1}\Pi^h_\Dx(\hat{Y}^h_{j})\Big\|_{\infty} \le 1.
$$
Then
\begin{align*}
  |u^{h}_{0,l}(\cdot,y) - u^h_0(\cdot,y)|_{\ell^\infty} &\le \E\bigg[\Big\|\prod_{j=0}^{N-1}\Pi^h_\Dx(\hat{Y}^h_{j})\Big\|_{\infty}  | f_l(\cdot,\hat{Y}^h_{N}) - f(\cdot,\hat{Y}^h_{N})|_{\ell^\infty}\bigg] \\
  &\le \E[|f_l(\cdot,\hat{Y}^h_{N}) - f(\cdot,\hY^h_{N})|_{\ell^\infty}]. 
\end{align*}
Now proceeding as for term $I$, one can show
\begin{equation}\label{estim_III}
  III \le \sup_{x\in\R,y\in[0,M]}|f(x,y)-f_l(x,y) |+ C\|(1+(\hY^h_{N})^L)\|_{L^p(\Omega)} \left(\frac{E[\hY^h_{N}]}{M}\right)^{\frac{p-1}{p}}.
\end{equation}
Finally, for every $\varepsilon>0$, thanks to the boundedness of the moments of $Y^{0,y}_T$ and the uniform (in $N$) boundedness of all the moments of $\hY^h_{N}$,  we can choose $M>0$ big enough to guarantee that the second term in the right-hand sides of \eqref{estim_I} and \eqref{estim_III} are less than $\varepsilon/5$. Chosen $M$, thanks to the uniform continuity of $f$, we can take $l$ big enough to guarantee $\sup_{x\in\R,y\in[0,M]}|f-f_l|\le \varepsilon/5$  and $N_0\in\N^*$, $\delta>0$ such that for every $N\ge N_0$ and $\Dx<\delta$ one has $C_lT(h+\Dx)<\varepsilon/5$, and so 
$$
|u(0,\cdot,y) - u^h_0(\cdot,y)|_{\ell^\infty} \le \frac{2\varepsilon}{5}+\frac{\varepsilon}{5}+\frac{2\varepsilon}{5} = \varepsilon.
$$
\end{proof}

\bibliographystyle{abbrv}
\bibliography{Bib_visco}

\begin{thebibliography}{10}

\bibitem{Bourbaki66}
N.~Bourbaki.
\newblock {\em Elements of Mathematics: General Topology, Part 2}.
\newblock Springer / Addison-Wesley Educational Publishers, Inc., 1966.

\bibitem{Brezis2010book}
H.~Brezis.
\newblock {\em Functional Analysis, Sobolev Spaces and Partial Differential Equations}.
\newblock Universitext. Springer New York, 2010.

\bibitem{BCT}
M.~Briani, L.~Caramellino, and G.~Terenzi.
\newblock Convergence rate of {M}arkov chains and hybrid numerical schemes to jump-diffusion with application to the {B}ates model.
\newblock {\em SIAM J. Numer. Anal.}, 59(1):477--502, 2021.

\bibitem{BCTZ}
M.~Briani, L.~Caramellino, G.~Terenzi, and A.~Zanette.
\newblock Numerical stability of a hybrid method for pricing options.
\newblock {\em International Journal of Theoretical and Applied Finance}, 22(07):1950036, 2019.

\bibitem{BCZ}
M.~Briani, L.~Caramellino, and A.~Zanette.
\newblock A hybrid approach for the implementation of the {H}eston model.
\newblock {\em IMA J. Manag. Math.}, 28(4):467--500, 2017.

\bibitem{BCZ2}
M.~Briani, L.~Caramellino, and A.~Zanette.
\newblock A hybrid tree/finite-difference approach for heston--hull--white-type models.
\newblock {\em Journal of Computational Finance}, 2017.

\bibitem{CMA17}
A.~Canale, R.~M. Mininni, and A.~Rhandi.
\newblock Analytic approach to solve a degenerate parabolic pde for the heston model.
\newblock {\em Mathematical Methods in the Applied Sciences}, 40(13):4982--4992, 2017.

\bibitem{CPD12}
C.~Costantini, M.~Papi, and F.~D'Ippoliti.
\newblock Singular risk-neutral valuation equations.
\newblock {\em Finance and Stochastics}, 16:249--274, 2012.

\bibitem{CI90}
M.~G. Crandall and H.~Ishii.
\newblock {The maximum principle for semicontinuous functions}.
\newblock {\em Differential and Integral Equations}, 3(6):1001 -- 1014, 1990.

\bibitem{CIL92}
M.~G. Crandall, H.~Ishii, and P.-L. Lions.
\newblock user's guide to viscosity solutions of second order partial differential equations, 1992.

\bibitem{ET2010}
E.~Ekström and J.~Tysk.
\newblock The black–scholes equation in stochastic volatility models.
\newblock {\em Journal of Mathematical Analysis and Applications}, 368(2):498--507, 2010.

\bibitem{Erdos1945}
P.~Erd{\"o}s.
\newblock Some remarks on the measurability of certain sets.
\newblock {\em Bulletin of the American Mathematical Society}, 51(10):728--731, 1945.

\bibitem{FS93book}
W.~H. Fleming, H.~M. Soner, H.~M. Soner, D.~A. Mathematics, F.~Fleming, and S.~Soner.
\newblock Controlled markov processes and viscosity solutions.
\newblock 1992.

\bibitem{AF_book}
A.~Friedman.
\newblock {\em Partial differential equations of parabolic type}.
\newblock Courier Dover Publications, 2008.

\bibitem{Heston}
S.~L. Heston.
\newblock A closed-form solution for options with stochastic volatility with applications to bond and currency options.
\newblock {\em Rev. Financ. Stud.}, 6(2):327--343, 1993.

\bibitem{PP2015}
S.~Pagliarani and A.~Pascucci.
\newblock The exact taylor formula of the implied volatility.
\newblock {\em Finance and Stochastics}, 21:661 -- 718, 2015.

\end{thebibliography}

\end{document}